\documentclass[11pt]{amsart}
\usepackage[toc]{appendix}
\usepackage{amsfonts,amssymb,amsmath,amsthm,latexsym,graphics,epsfig,amsfonts}
\usepackage{verbatim,enumerate,array,booktabs,color,bigstrut,prettyref,tikz-cd}
\usepackage{multirow}
\usepackage[all]{xy}
\usepackage{hyperref}
\usepackage[OT2,T1]{fontenc}
\usepackage{mathtools}
\usepackage{caption}
\usepackage{longtable}
\usepackage{mathtools}
\usepackage{tabularray}
\UseTblrLibrary{amsmath,varwidth}
\usepackage{tabularx}
\usepackage{longtable}
\usepackage{arydshln}
\usepackage[normalem]{ulem}
\usepackage{cancel}

\usepackage{amsmath, amssymb, amsthm}
\usepackage{hyperref}

\newcommand{\Kd}{\operatorname{K}}
\newcommand{\kI}{\operatorname{I}}
\newcommand{\kII}{\operatorname{II}}
\newcommand{\kIII}{\operatorname{III}}
\newcommand{\kIV}{\operatorname{IV}}
\newcommand{\Emin}{\operatorname{\mathcal E}}
\newcommand{\Emind}{\operatorname{\mathcal E}^d}

\usepackage{amsfonts,amssymb,amsmath,amsthm,latexsym,graphics,epsfig,amsfonts}
\usepackage{verbatim,enumerate,array,booktabs,color,bigstrut,prettyref,tikz-cd}
\usepackage[all]{xy}
\newcommand{\cond}[1]{\mathbf{#1}}

\newcommand{\Mod}[1]{\ (\mathrm{mod}\ #1)}

\makeatletter
\newcommand*\circled[2][1.6]{\tikz[baseline=(char.base)]{
    \node[shape=circle, draw, inner sep=1pt, 
        minimum height={\f@size*#1},] (char) {\vphantom{WAH1g}#2};}}
\makeatother
\newtheorem{theorem}{Theorem}
\newtheorem{proposition}{Proposition}
\newtheorem{lemma}{Lemma}

\newcommand{\Q}{\mathbb Q}
\newcommand{\Z}{\mathbb Z}
\newcommand{\eclabel}[1]{\href{https://www.lmfdb.org/EllipticCurve/Q/#1}{\texttt{#1}}}

\begin{document}

\author{Enrique Gonz\'alez-Jim\'enez}
\address{Departamento de Matem\'aticas, Universidad Aut\'onoma de Madrid, Madrid, Spain
}
\email{enrique.gonzalez.jimenez@uam.es}

\author{Joan-C. Lario}
\address{Departament de Matem\`atiques, 
Universitat Polit\`ecnica de Catalunya\\
Barcelona, Catalunya }
\email{joan.carles.lario@upc.edu}

\thanks{The first author is supported by Grant PID2022-138916NB-I00, and the second author by Grant PID2022-136944NB-I00; both funded by MCIN/AEI/10.13039/501100011033 and by ERDF A way of making Europe.}

\title{Faltings elliptic curves in twisted $\mathbb{Q}$-isogeny classes}
\date{\today}

\begin{abstract}
Let $G$ be the graph attached to the $\Q$-isogeny class of an elliptic curve defined over~$\Q$: that is, a vertex for every elliptic curve defined over $\Q$ in the isogeny class, and edges in correspondence with the  prime degree rational isogenies 
between them. Stevens \cite{S} shows that there is a unique elliptic curve in $G$ with minimal Faltings height. We call  this curve the Faltings elliptic curve in $G$.

For every  square-free integer $d$, we consider the graph~$G^d$
attached to the twisted elliptic curves in $G$ 
by the quadratic character of $\Q(\sqrt{d})$.
It turns out that $G$ and $G^d$ are canonically isomorphic as abstract graphs
(the isomorphism identifies the vertices with equal $j$-invariant). 
In this paper we determine which vertex is the Faltings elliptic curve in $G^d$. We also obtain the probability of a vertex in $G$ to be the Faltings elliptic curve in~$G^d$. 
It turns out that this probability depends on the $p$-adic valuations of rational values of certain modular functions.
\end{abstract}

\maketitle

\vskip 0.2truecm


\section{Introduction} 
The main object of this paper is
the $\Q$-isogeny classes of elliptic curves defined over the rationals and their associated isogeny graphs. For an elliptic curve $E$ over $\mathbb{Q}$, we shall denote by $\operatorname{Isog}(E)$ its $\mathbb{Q}$-isogeny class, and by $G=G(E)$ the isogeny graph attached to it: that is, one vertex for each elliptic curve in $\operatorname{Isog}(E)$ and an edge for every prime degree rational isogeny connecting two vertices.

Given a $\Q$-isogeny class of an elliptic curve defined over $\Q$, it is natural to ask if there is any distinguished elliptic curve in the isogeny class.
First, Mazur and Swinnerton-Dyer \cite{MS} proposed the so-called {\it strong} curve which is an optimal quotient of the Jacobian of the modular curve $X_0(M)$, where $M$ is the conductor of the isogeny class.
Later, Stevens \cite{S} suggested that it is better to consider the elliptic curve which is an optimal quotient of the Jacobian of the modular curve~$X_1(M)$. In both cases, the Manin constant plays a role, and the Stevens proposal seems to be more intrinsically arithmetic due to the intervention of N\'eron models, \'etale isogenies, and Faltings heights (independently of modularity parametrizations). In this paper, we define the Faltings elliptic curve as the one with minimal Faltings height in the isogeny class. Conjecturally, the Faltings elliptic curve is an optimal quotient of the Jacobian of $X_1(M)$. See \cite{S} and \cite{Vatsal}.

The different possibilities for the $\Q$-isogeny graphs of elliptic curves over $\Q$ are classified into different types:
$$
\begin{tblr}{l}
L_1, \\
L_2(p) 
\text{ for $p$ in $\{ 2,3,5,7,11,13,17,19,37,43,67,163\}$}, \\
L_3(9), L_3(25), L_4, T_4, T_6, T_8, \\
R_4(N) \text{ for $N$ in $\{ 6,10,14,15,21\}$}, \\
R_6, S_8.
\end{tblr}
$$

The subscripts denote the number of vertices of the graph, the letter indicates the shape of the isogeny graph ($L$ line, $T$ tetrahedron, $R$ rectangular, $S$ special), and the level in parentheses refers to the maximal isogeny 
degree of a path 
in the graph (absence of level means isogenies of degree $2$ or~$3$).

This classification is based on the knowledge of the non-cuspidal rational points of the modular curves $X_0(N)$. First, Mazur \cite{M} dealt with the case of prime degree rational isogenies of elliptic curves defined over~$\mathbb{Q}$. The  classification for arbitrary degrees, prime or composite, was completed thanks to the work of Fricke, Kenku, Klein, Kubert, Ligozat, Mazur and Ogg, among others. In particular, 
Kenku shows that there are at most 8 elliptic curves in each isogeny class over $\mathbb{Q}$. We refer to Chiloyan--Lozano-Robledo \cite{C} for more details on the classification and the sketch of its proof.

\begin{longtblr}
[label={types}, caption={Types of $\Q$-isogeny graphs}]
{cells = {mode=imath},hlines,vlines,measure=vbox,
colspec=cll,rowhead=1}
\text{type} &\SetCell[c=2]{l}  \text{$\Q$-isogeny graph} & \\
L_2(p) &
\begin{tikzcd}[ampersand replacement=\&] 
E_1 \arrow[dash]{r}{p} \& E_p 
\end{tikzcd}
&
p=2,3,5,7,11,13,17,19,37,43,67,163
\\
L_3(p^2) &
\begin{tikzcd}[ampersand replacement=\&] 
E_1 \arrow[dash]{r}{p} \& E_p  \arrow[dash]{r}{p} \& E_{p^2}   
\end{tikzcd}
&
p=3,5
\\
L_4 &\SetCell[c=2]{l}
\begin{tikzcd}[ampersand replacement=\&] 
E_1 \arrow[dash]{r}{3} \& E_{3}\arrow[dash]{r}{3} \& E_{9}\arrow[dash]{r}{3} \& E_{27}
\end{tikzcd}
&
\\
R_4(pq) &
\begin{tikzcd}[ampersand replacement=\&] 
E_1 \arrow[dash]{r}{q} 
    \arrow[dash]{d}{p} \& 
    E_{q}  \arrow[dash]{d}{p} \\
 E_p \arrow[dash]{r}{q} \& E_{pq}   
\end{tikzcd}
&
 (p,q)=(2, 3), (2, 5), (2, 7), (3, 5), (3, 7)
\\
R_6 & \SetCell[c=2]{l}
\begin{tikzcd}[ampersand replacement=\&] 
  E_1 \ar[dash,swap,d,"2"] \ar[dash,r,"3"]   \& E_3\ar[dash,d,"2"] \ar[dash,r,"3"]  \& E_9  \ar[dash,d,"2"]  \\
  E_2 \ar[dash,swap,r,"3"] \& E_{6} \ar[dash,swap,r,"3"] \& E_{18}  
\end{tikzcd}
&
\\
T_4 & \SetCell[c=2]{l}
 \scalebox{.8}{
 \begin{tikzcd}[ampersand replacement=\&] 
\& E_1 \ar[dash,d,"2"] \& \\
\& E_2 \ar[dash,ld,swap,"2"] \ar[dash,rd,"2"] \& \\
E_4   \& \&  E_{12}   
\end{tikzcd}
}
&
\\
T_6 & \SetCell[c=2]{l}
\scalebox{.8}{
\begin{tikzcd}[ampersand replacement=\&] 
E_1 \ar[dash,dr,"2"] \& \& \& E_8\\
\& E_2 \ar[dash,r,"2"] \& E_4  \ar[dash,ur,"2"]  \ar[dash,dr,"2"]\& \\
E_{12}    \ar[dash,ur,"2"] \&  \&  \& E_{22}    
\end{tikzcd}
}
&
\\
T_8 &  \SetCell[c=2]{l}
\scalebox{.8}{
\begin{tikzcd}[ampersand replacement=\&] 
\& E_{21} \ar[dash,d,"2"] \& \& E_{81} \ar[dash,d,"2"] \& \\
\& E_{2} \ar[dash,swap,dl,"2"] \ar[dash,dr,"2"] \& \& E_{8} \ar[dash,dl,swap,"2"] \ar[dash,dr,"2"] \& \\
 E_1 \&  \& E_4 \ar[dash,d,"2"] \& \& E_{16}  \\
 \&  \& E_{41} \& \&    
\end{tikzcd}
}
&
\\
S_8 & \SetCell[c=2]{l}
\scalebox{.85}{
 \begin{tikzcd}[ampersand replacement=\&] 
  \& E_1 \ar[dash,swap,d,"2"] \ar[dash,r,"3"]   \& E_3 \ar[dash,d,"2"] \& \\
  \& E_2 \ar[dash,swap,dl,"2"] \ar[dash,swap,dr,"2",pos=1/6] \ar[dash,r,"3"]   \& E_6  \ar[dash,dl,"2",pos=1/6] \ar[dash,dr,"2"]  \& \\
    E_{21} \ar[dash,r,"3"] \&   E_{12} \& E_ 4 \ar[dash,r,"3"]    \& E_{31}  
\end{tikzcd}
}
&
\end{longtblr}



For every square-free integer $d$, we denote by $E^d$ the quadratic twist of the elliptic curve $E$ by~$\mathbb{Q}(\sqrt{d})$, and  let $G^d=G(E^d)$ be the graph attached to the $\mathbb{Q}$-isogeny class of $E^d$.
It turns out that $G$ and $G^d$ are canonically isomorphic as labeled graphs (the isomorphism identifies the vertices with equal $j$-invariant). Our aim in this paper is to discuss which vertex in $G$ is the Faltings elliptic curve in the twisted graphs $G^d$. In particular, we obtain the probability of a vertex to be the Faltings elliptic curve in $G^d$ in terms of $d$ and the elliptic curve $E$. 
As we shall see, in many cases, this probability depends on the $p$-valuations of rational values of modular functions. See Theorem \ref{MainTheorem} in Section \ref{sec_MainTheorem}.

The plan of the paper is as follows. In Section \ref{terminology} we fix the terminology
and review several general facts. In Section
\ref{strategy} we describe the strategy to achieve our goal. In Section \ref{examples} we present the results and provide the details for the $\mathbb{Q}$-isogeny graphs  of types $L_3(9)$ and $L_2(11)$.
Section \ref{sec_MainTheorem} contains the results for all the remaining cases
of $\Q$-isogeny graphs. 
Our computations have been performed using a combination of \verb|Magma| \cite{magma}, \verb|Mathematica| \cite{mathematica}, and \verb|SageMath| \cite{sage}. We refer to the supplementary material at the GitHub repository \cite{githubrepo} for the precise data and related information for all cases.

To finish this introduction, we include a result that will be useful in order to compute the probability of a vertex to be the Faltings curve in the twisted $\Q$-isogeny classes.

\begin{lemma}\label{analyticNT}
The proportion of all square-free integers that are divisible by a prime $p$ is $
\displaystyle{\frac{1}{1+p}}$.
\end{lemma}

\begin{proof}
Brown \cite[Theorem 1]{Brown} shows that if $P$ and $Q$ are disjoint sets of prime numbers with $P$ finite, then
the proportion of all numbers which are square-free and divisible by all of the primes in $P$ and by none of the primes in $Q$ is
$$
\displaystyle \frac{6}{\pi^2}\prod_{p\in P}\frac{1}{1+p}\prod_{q\in Q}{\frac{q}{1+q}}.
$$
If $P=Q=\emptyset$ we obtain that the proportion of square-free numbers in $\Z$ is $\frac{6}{\pi^2}$. On the other hand, if  $P=\{p\}$ and $Q=\emptyset$ we obtain that the proportion of those that are divisible by $p$ is $\frac{6}{\pi^2}\frac{1}{1+p}$. 
Thus we obtain the desired result.
\end{proof}

\section{Preliminaries}\label{terminology}

\noindent {\bf Terminology.} Let $E$ be an elliptic curve defined over $\Q$ given by a Weierstrass equation
$$
E\colon y^2+a_{1}xy+a_{3}y=x^3+a_{2}x^2+a_{4}x+a_{6}\,.
\label{Weiers_equation}
$$
We shall denote its signature  
by $\operatorname{sig}(E)=(c_4,c_6,\Delta)$, where the 
$c$-invariants \( c_4 \), \( c_6 \), and the discriminant $\Delta$ are obtained from the defining equation of $E$ through the standard formulas (see~\cite{Sil}, Chapter III.1). For every prime $p$, we shall denote by
$\operatorname{sig}_p(E)=(v_p(c_4),v_p(c_6),v_p(\Delta))$
where $v_p$ denotes the standard $p$-adic valuation.
Additionally, we consider:
$$
\begin{tblr}{ll}
j  =\displaystyle{\frac{c_4^3}{\Delta}}& \text{$j$-invariant of $E$,} \\[3pt]
\omega=
\displaystyle{\frac{dx}{2y+a_1x+a_3}}
& \text{invariant differential of $E$,}\\[5pt]
\Lambda =
\lambda \langle 1,\tau \rangle &  \text{period lattice of $E$,}
\end{tblr}
$$
with
$\lambda\in \mathbb{C}^*$ and $\tau$
in the Poincaré half-plane $\mathbb{H}$. Via the uniformization theorem, the complex points of the elliptic curve $E$ can be identified with the complex torus $\mathbb{C}/\Lambda$. One has
$j=j(\tau)$ through the Klein modular function. For future use, 
we shall denote by $\operatorname{vol}(\Lambda)$ the volume of the period lattice of $E$; that is, $\operatorname{vol}(\Lambda)=|\lambda|^2\operatorname{Im}(\tau)$.
 
\vskip 0.3truecm 

\noindent {\bf Isomorphisms.} An elliptic curve $E'/\Q$ 
is isomorphic to $E/\Q$ if it can be obtained from $E$  
through a Weierstrass change of variables:
$$
\begin{array}{l@{\,=\,}l}
x & u^2 x'+ r \\[4pt]
y & u^3y'+u^2sx'+t
\end{array}
$$
with $u,r,s,t$ in $\mathbb{Q}$, $u\neq 0$. Following the above notations, one has:
$$
\begin{array}{lclcl}
u^4 c_4' = c_4 & & j' = j & &\\[4pt]
u^6 c_6' = c_6 & & u^{-1} \omega' = \omega & & \\[4pt]
u^{12}\Delta' = \Delta & & u^{-1} \Lambda' = \Lambda & & |u|^{-2}\operatorname{vol}(\Lambda') = \operatorname{vol}(\Lambda)\,.
\end{array}
$$


\vskip 0.2truecm 

\noindent {\bf Isogenies.}
Let \(E_1\) and \(E_2\) be elliptic curves defined over $\mathbb{Q}$ such that there is
an isogeny  
\(\phi: E_1 \to E_2\) defined over $\Q$.
   The induced map $\phi^*$ on tangent spaces 
    yields to
    
    $$
    \phi^*(\omega_2) = c \, \omega_1
    $$

\vskip 0.3truecm

\noindent for some $c\in\Q^*$, where $\omega_i$ denotes
    the invariant differential of $E_i$. The constant $c$ is called the scaling factor of the isogeny. We shall say that the isogeny $\phi$ is normalized when $c=1$. 

\begin{lemma}
\label{normal}
Let $\varphi\colon  \mathbb C/\Lambda\longrightarrow  \mathbb C/\Lambda'$ be a normalized isogeny of prime degree~$p$, then  $\operatorname{vol}(\Lambda)=p\operatorname{vol}(\Lambda')$.
\end{lemma}

\begin{proof}
Write $\Lambda=\lambda \langle 1, \tau \rangle$ with $\lambda,\tau\in \mathbb C^*$ and $\operatorname{Im}(\tau)>0$. Since $\phi$ is normalized, we must have $\varphi(z)=z$; that is, 
$\Lambda\subset \Lambda'$ has index $p$. Hence, one has either $\Lambda'=\lambda/p\langle 1,p\tau\rangle$ or $\Lambda'=\lambda\langle 1,(\tau+k)/p\rangle$ for some $k=0,\dots,p-1$. Then the result follows directly by computing the volumes of the two lattices.
\end{proof}

    Given a finite subgroup \(H \subset E_1\), there exists a normalized isogeny \(\phi: E_1 \to E_2\) with \(\ker(\phi) = H\) (see \cite[III.4]{Sil}). The curve \(E_2\) and the isogeny \(\phi\) can be explicitly constructed using Vélu's formulas \cite{Velu}.

\vskip 0.3truecm 

\noindent {\bf Twists.} 
For every square-free integer $d$, we denote by $E^d$ the quadratic twist of the elliptic curve $E$ by $\Q(\sqrt{d})$. If $\operatorname{sig}(E)=(c_4,c_6,\Delta)$, then
$\operatorname{sig}(E^d)=(d^{2}c_4,d^{3}c_6,
d^{6}\Delta)$.

Let $\Emin$ denote a minimal model of~$E$.
Then,  $\operatorname{sig}(\Emin)=
(c_4(\Emin),c_6(\Emin),\Delta(\Emin))$ 
is composed of the three integers and the
discriminant $\Delta(\Emin)$ is minimal among the discriminants of all $\Q$-isomorphic elliptic curves to $E$. Observe that the signature $\operatorname{sig}(\Emind)$ does not need to be the signature of a minimal model of the quadratic twist $E^d$.
The following result is due to Pal and relates the signatures and the Néron lattices of the elliptic curves $E$ and $E^d$. In what follows, by Néron lattice we mean the period lattice attached to a minimal model.

\begin{proposition}[Pal \cite{Pal}]
Let $E$ be an elliptic curve over $\mathbb{Q}$ and 
let $\Emin$ denote a minimal model of~$E$.
Let $\Lambda(\mathcal E)$ be the Néron lattice of $E$. For every square-free integer $d$,
the N\'eron lattice $\Lambda^d$ of the quadratic twist $E^d$ is given by $$
\Lambda^d = 
\displaystyle{\frac{u(\Emind)}{\sqrt{d}}} \Lambda(\mathcal E)
$$
where $u(\Emind)=\prod_p u_p(\Emind)$ is the rational number obtained as a product over each prime~$p$ according to the following table:

\vskip 0.1truecm

\begin{longtblr}
[
label = {pal},
caption = {Pal values $u_p(\Emind)$}
]
{cells = {mode=imath},hlines,measure=vbox,colspec  = ccl,
vlines = {1-19}{solid}}
\SetCell[c=3]{c}  p\ne 2\\
d & u_p(\Emind)  &  \text{conditions}  \\
 d\not\equiv 0\,(p) & 1 &  \\
  \SetCell[r=2]{c}    d\equiv 0\,(p) & 1 &  \Kd_{p}(E)=\kI_n\, (n\ge 0), \kII,\kIII,\kIV \\
  & p& \Kd_{p}(E)=\kI^*_n\, (n\ge 0), \kII^*,\kIII^*,\kIV^* \\
\SetCell[c=3]{c} 
 \pagebreak
\SetCell[c=3]{c}  p= 2\\
d & u_p(\Emind)  &  \text{conditions}  \\
 d\equiv 1\,(4)& 1 &  \\
  \SetCell[r=6]{c}    d\equiv 2\,(4) &  2^{-1} & \operatorname{sig}_2(\Emin)=(0,0,\ge 0) \\
  & 4 & \operatorname{sig}_2(\Emin)=(6,9,\ge 18)\,\,\text{and}\,\, c_6(\Emin)2^{-10}d\equiv -1\,(4)\\ 
  &  \SetCell[r=3]{c} 1 & v_2(c_4(\Emin))=4\,\, \text{or}\,\, 5 \\
  & & v_2(c_6(\Emin))=3,5, \,\, \text{or}\,\,  7\\
  & & \operatorname{sig}_2(\Emin)=(\ge 6,6,6)\,\,\text{and}\,\, c_6(\Emin)2^{-7}d\equiv -1\,(4)\\ 
  & 2 & \text{otherwise}\\ 
  \SetCell[r=5]{c}   d\equiv 3\,(4) &  \SetCell[r=2]{c} 2^{-1}  &  \operatorname{sig}_2(\Emin)=(0,0,\ge 0) \\
   & &  \operatorname{sig}_2(\Emin)=(\ge 4,3,0) \\
   &  \SetCell[r=2]{c} 2  &  \operatorname{sig}_2(\Emin)=(4,6,\ge 12) \\
   & &  \operatorname{sig}_2(\Emin)=(\ge 8,9,12) \\
     & 1 & \text{otherwise}\\ 
\end{longtblr}
\vskip 0.3truecm

\noindent Here, $\Kd_p(E)$ denotes the Kodaira symbol of $E$ at $p$, and  $\operatorname{sig}_p(\mathcal E)=(v_p(c_4(\Emin)) ,
v_p(c_6(\Emin)),v_p(\Delta(\Emin)))$ is the $p$-signature of $\mathcal E$.
\end{proposition}


\noindent {\bf Minimal models and Pal values.}  
Given an elliptic curve $E/\Q$ defined by a Weierstrass equation and a prime $p$, 
Tate's algorithm \cite{Tate} produces a Weierstrass change $u,r,s,t\in \Q$, $u\ne 0$, giving rise to a $p$-minimal equation of $E$: that is, a Weierstrass equation with integer coefficients and 
discriminant with minimal $p$-valuation.
The algorithm describes as well the special fibers of the N\'eron model of $E$ at $p$ tagged by the Kodaira symbols $\Kd_p(E)$: $\kI_0$, $\kI_n$, $\kII$, $\kIII$, $\kIV$, $\kI_0^*$, $\kI_n^*$, $\kIV^*$, $\kIII^*$, $\kII^*$. Papadopulus~\cite{Papa} employs Tate's algorithm to examine how to determine a $p$-minimal model of $E$ and $\Kd_p(E)$  directly from the $p$-adic signature $\operatorname{sig}_p(E)$. For \( p \geq 5 \) this can be achieved, but for \( p = 2 \) and \( p = 3 \) one needs the analysis in \cite{Papa} plus the refinement given by Browkin and Davies \cite{poland} to provide the extra conditions that permit distinguishing between different $\Kd_p(E)$ that may arise from the same $p$-signature. 

For future use, in the next tables we compile the information in ~\cite{Papa} and \cite{poland} along with Proposition~\ref{pal} that allows us to determine $p$-minimal models $\Emin$ of $E$ and for the $d$-quadratic twist
$\Emind$. More precisely, Tables \ref{PapaNo23}, \ref{Papa3}, and \ref{Papa2} provide the Kodaira symbols \( \Kd_p(E) \) at \( p \), as well as the Pal values \( u_p(\mathcal{E}^d) \) for square-free integers \( d \). Whenever a row displays the expression $\oplus\, (u=p)$,
it means that \( E \) is not \( p \)-minimal and one needs to transform its Weierstrass equation with a change given by \( u = p \) and then to
restart searching in the table with the new \( p \)-adic signature.

Table \ref{Papa3} (for $p=3$) and Table \ref{Papa2} (for $p=2$) include boldfaced labels to
indicate that an extra condition must be satisfied. 
We use the logical negation symbol \( \lnot \) to mark  the opposite condition, and the logical conjunction symbol \( \land \) to combine conditions.

\


\noindent \fbox{Case $p\ne 2,3$}
In these cases, the $p$-signature $\operatorname{sig}_p(E)$ determines the Kodaira symbol $\Kd_p(E)$, a $p$-minimal model $\Emin$ of $E$,  and  a $p$-minimal model of $\Emind$ as well.
We refer to  \cite[\S I, Tableau~I]{Papa}.
\newpage
\begin{longtblr}
[
label = {PapaNo23},
caption = {Cases $p\ne 2,3$}
]
{cells = {mode=imath},hlines,vlines,measure=vbox,colspec  = cccc}
  p &\text{Kodaira}   & \SetCell[c=2]{c} u_p(\Emind)\\
\operatorname{sig}_{p}(E)  &\Kd_{p}(E)  & d \equiv 0\,(p) & d\not \equiv 0\,(p)  \\
(0,\ge 0,0) & \kI_0 & 1& 1\\
(\ge 0,0,0) & \kI_0 & 1& 1\\
(\ge 0,0,n) & \kI_n & 1& 1\\
(\ge 1,1,2) & \kII & 1& 1\\
(1,\ge 2,3) & \kIII & 1& 1\\
(\ge 2,2,4) & \kIV & 1& 1\\
(2,\ge 3,6) & \kI_0^* & p& 1\\
(\ge 2, 3,6) & \kI_0^* & p& 1\\
(2, 3,6+n) & \kI_n^* & p& 1\\
(\ge 3,4,8) & \kIV^* & p& 1\\
(3,\ge 5,9) & \kIII^* & p& 1\\
(\ge 4,5,10) & \kII^* & p& 1\\
(\ge 4,\ge 6,\ge 12) &\SetCell[c=3]{c}   \oplus\, (u=p)   \\
\end{longtblr}

\noindent \fbox{Case $p=3$}
It turns out that the same signature $\operatorname{sig}_3(E)$ can occur in several Kodaira symbols. This ambiguity is solved in \cite[\S II.2, Tableau II]{Papa}  by using the following conditions:

$$
\begin{array}{cl}
\cond{3_a}:&  
(c_6/3^3)^2+2-3(c_4/3^2) \equiv 0\!\pmod{9}.\\[4pt]
\cond{3_b}:&  (c_6/3^6)^2+2-3(c_4/3^4) \equiv 0\!\pmod{9}.\\[4pt]
\end{array}
$$

\noindent Observe that conditions $\cond{3_a}$ and
$\cond{3_b}$
correspond to $P_2$ and $P_5$, respectively, in \cite[\S II.2]{Papa}.

\begin{longtblr}
[
label = {Papa3},
caption = {Case $p=3$}
]
{cells = {mode=imath},hlines,vlines,measure=vbox,colspec  = ccccc,rowhead = 2}
  p=3 & \SetCell[r=1, c=2]{c} \text{Kodaira}  & & \SetCell[c=2]{c} u_3(\Emind)\\
\operatorname{sig}_{3}(E)  &\Kd_{3}(E) & \text{condition}  & d \equiv 0\,(3) & d\not \equiv 0\,(3)  \\
(0,0,n) & \kI_n &   & 1& 1\\
(1,\ge 3,0) & \kI_0 &   & 1& 1\\
\SetCell[r=2]{c}  (\ge 2,3,3) & \kIII & \cond{3_a} & \SetCell[r=2]{c}  1& \SetCell[r=2]{c}  1\\
                                     & \kII &  \neg\,\,\cond{3_a} & &  \\
(2,3,4) & \kII &   & 1& 1\\
(2,3,5) & \kIV &   & 1& 1\\
(2,3,6+n) & \kI_n^* &   & 3& 1\\
(2,4,3) & \kII &   & 1& 1\\
(2,\ge 5,3) & \kIII &   & 1& 1\\
(\ge 3,4,5) & \kII &   & 1& 1\\
(3,5,6) & \kIV &   & 1& 1\\
(3,\ge6,6) & \kI_0^* &   & 3& 1\\
(\ge 4,5,7) & \kIV &   & 1& 1\\
\SetCell[r=2]{c}  (\ge 4,6,9) & \kIII^*& \cond{3_b} & \SetCell[r=2]{c}  3& \SetCell[r=2]{c}  1\\
                                     & \kIV^* &  \neg\,\,\cond{3_b} & &  \\
  (4,6,10) &  \kIV^* & & 3& 1\\
  (4,6,11) &  \kII^* & & 3& 1\\
  (4,6,12+n) &  \kI_n & \oplus\, (u=3)  & 1& 1\\
  (4,7,9) &  \kIV^* & & 3& 1\\
  (4,\ge 8,9) &  \kIII^* &  & 3& 1\\
  (\ge 5,7,11) &  \kIV^* &  & 3& 1\\
  (5,8,12) &  \kII^* &  & 3& 1\\
  (5,\ge 9,12) &  \kI_0 & \oplus\, (u=3)  & 1& 1\\
  (\ge 6,8,13) & \kII^* &   & 3& 1\\
(\ge 6,\ge 9,\ge 15) &\SetCell[c=4]{c}    \oplus\, (u=3)   \\
\end{longtblr}

\vskip 1truecm 

\noindent \fbox{Case $p=2$} In this case, more conditions are required to determine \( K_2(E) \) and the $p$-minimal models. 
The ambiguities are solved in \cite{Papa} and \cite{poland} by using the following conditions.  Set $A=-c_4/48$ and $B=-c_6/864$, and 
define the polynomials

$$
\begin{array}{l}
\Psi_2(x)=x^3+Ax+B\,,\\[8pt]
\Psi_3(x)=3 x^4+6Ax^2+12Bx-A^2\,.
\end{array}
$$

\noindent The conditions are:
\begin{itemize}
\item[]$\cond{2_a}$:  
$\Bigl(
A\equiv 1\!\pmod{4}
\text{ and } B\equiv 0,1\!\pmod{4}
\Bigr)$ 
or
$
\Bigl(
A\not\equiv 1\!\pmod{4} 
\text{ and }
B\equiv 2,3\!\pmod 4
\Bigr)$. 
\item[] $\cond{2_b}$: $\Psi_3(A)\not\equiv 0\!\pmod{8}$.
\item[]$\cond{2_c}$: Either $\Psi_3(x)\equiv 0\!\pmod{32}$ has no solution or 
if
$\Psi_3(r)\equiv 0\!\pmod{32}$ then 
$\Psi_2(r)\equiv 1,8,9,12\!\pmod{16}$. 
\item[]$\cond{2_d}$:  If $\Psi_3(r)\equiv 0\!\pmod{32}$, then $r\equiv 1,2\!\pmod{4}$.
\item[] $\cond{2_e}$: $c_4/2^6\equiv 3\!\pmod{4}.$
\item[] $\cond{2_f}$: $c_6/2^6\equiv 1\!\pmod{4}.$
\item[] $\cond{2_g}$: $c_6/2^9\equiv 3\!\pmod{4}$.
\end{itemize}
Observe that conditions $\cond{2_a}$, $\cond{2_b}$, $\cond{2_c}$ and $\cond{2_d}$ arise from Propositions 1--7 of \cite[\S II.3]{Papa}. Conditions $\cond{2_e}$, $\cond{2_f}$, and $\cond{2_g}$ do not appear in \cite{Papa}.
However, condition $\cond{2_e}$ is required when \( \operatorname{sig}_2(E) = (6, \ge 9, 12) \) (cf. \cite[Lemma 2.9]{poland}). Also, conditions $\cond{2_f}$ and $\cond{2_g}$ are necessary to determine whether the cases \( \operatorname{sig}_2(E) = (4, 6, \geq 12) \) and \( \operatorname{sig}_2(E) = (\ge 8, 9, 12) \) are 2-minimal (cf. \cite[Lemma 2.2]{poland}), and in such cases, to 
determine the Kodaira symbol \( K_2(E) \) and to get a $2$-minimal model.

\begin{longtblr}
[label={Papa2},caption = {Case $p=2$}]
{cells = {mode=imath},hlines,vlines,measure=vbox,colspec  = cccccc,hspan=even, rowhead = 2}
  p=2 & \SetCell[r=1, c=2]{c} \text{Kodaira}  & & \SetCell[c=3]{c} u_2(\Emind)\\
\operatorname{sig}_{2}(E)  &\Kd_{2}(E) & \text{condition}  & d \equiv 1\,(4) & d\equiv 2\,(4) & d\equiv 3\,(4)  \\
(0,0,n) & \kI_n &   & 1& 2^{-1}& 2^{-1}\\
(\ge 4,3,0) & \kI_0 & & 1& 1 & 2^{-1}\\
\SetCell[r=3]{c}  (4,5,4) & \kII & \cond{2_a} & \SetCell[r=3]{c}  1& \SetCell[r=3]{c} 1  & \SetCell[r=3]{c}   1\\
                                     & \kIII &  \neg\,\,\cond{2_a}\, \land\, \cond{2_b} & & & \\
                                     & \kIV &  \neg\,\,\cond{2_a} \land \neg\,\,\cond{2_b} & & & \\
\SetCell[r=2]{c}  (4,\ge 6,6) & \kII & \cond{2_a} & \SetCell[r=2]{c} 1& \SetCell[r=2]{c} 1& \SetCell[r=2]{c} 1\\
                                     & \kIII &  \neg\,\,\cond{2_a}  & & & \\
(4,6,7) & \kII &   & 1& 1& 1\\
\SetCell[r=3]{c}  (4,6,8) & \kI_0^* & \cond{2_c} & \SetCell[r=3]{c} 1& \SetCell[r=3]{c} 1& \SetCell[r=3]{c} 1\\
                                     & \kI_1^*&  \neg\,\,\cond{2_c}\, \land\, \cond{2_d} & & & \\
                                     & \kIV^* &  \neg\,\,\cond{2_c} \land \neg\,\,\cond{2_d} & & & \\
(4,6,9) & \kI_0^*&   & 1& 1 & 1\\
\SetCell[r=2]{c}  (4, 6,10) & \kI_2^* & \cond{2_d} & \SetCell[r=2]{c} 1& \SetCell[r=2]{c} 1& \SetCell[r=2]{c} 1 \\
                                     & \kIII^* &  \neg\,\,\cond{2_d}  & & & \\
\SetCell[r=2]{c}  (4, 6,11) & \kI_3^* & \cond{2_d} & \SetCell[r=2]{c} 1& \SetCell[r=2]{c} 1& \SetCell[r=2]{c} 1\\
                                     & \kII^* &  \neg\,\,\cond{2_d}  & & & \\
\SetCell[r=2]{c}  (4, 6,12+n) & \kI_{n+4}^* & \cond{2_f} & 1& 1& 2 \\
                                     & \kI_n &  \neg\,\,\cond{2_f} \oplus\, (u=2)  & 1& 2^{-1} &2^{-1} \\
\SetCell[r=2]{c}  (5,5,4) & \kII & \cond{2_a} & 1& \SetCell[r=2]{c} 1&\SetCell[r=2]{c} 1 \\
                                     & \kIII &  \neg\,\,\cond{2_a}  & 1& & \\
(5,6,6) & \kII &   & 1& 1& 1\\
(\ge 6,6,6) & \kII &   & 1& \text{$1^*$ or $2^*$} & 1\\
(5,7,8) & \kIII &   & 1& 1& 1\\
(5,\ge 8,9) & \kIII &   & 1& 1& 1\\
\SetCell[r=2]{c}  (\ge 6,5,4) & \kII & \cond{2_a} & \SetCell[r=2]{c} 1&\SetCell[r=2]{c}  1&\SetCell[r=2]{c} 1 \\
                                     & \kIV &  \neg\,\,\cond{2_a}  & & & \\
\SetCell[r=2]{c}  (6,7,8) & \kI_0^*& \cond{2_c} & \SetCell[r=2]{c} 1& \SetCell[r=2]{c} 1&\SetCell[r=2]{c} 1 \\
                                     & \kI_1^*&  \neg\,\,\cond{2_c} & & & \\
(\ge 6,8,10) & \kI_0^*&   & 1& 2& 1\\
(6,9,13) & \kI_2^*&   & 1& 2&1 \\
 (6, 9,14+n), n<4 & \kI_{4+n}^* &   & 1& 2 & 1\\
(6, 9,18+n) & \kI_{n+8}^* &   & 1& \text{$4 ^*$ or $2^*$}  & 1\\
\SetCell[r=2]{c}  (6,\ge 9,12) & \kI_2^* & \cond{2_e} & \SetCell[r=2]{c} 1& \SetCell[r=2]{c}  2& \SetCell[r=2]{c}  1 \\
                                     & \kI_3^* & \neg \cond{2_e}  & & & \\
\SetCell[r=2]{c}  (\ge 7,7,8) & \kI_0^* & \cond{2_c} & \SetCell[r=2]{c} 1& \SetCell[r=2]{c} 1&\SetCell[r=2]{c} 1 \\
                                     & \kIV^* &  \neg\,\,\cond{2_c}  & & & \\
(7,9,12) & \kIII^*&   & 1& 2& 1\\
(7,10,14) & \kIII^*&   & 1& 2& 1\\
(7,\ge 11,15) & \kIII^*&   & 1& 2& 1\\
\SetCell[r=2]{c}  (\ge 8, 9,12) & \kII^* & \cond{2_g} & 1& 2& 2\\
                                     & \kI_0 &  \neg\,\,\cond{2_g} \oplus\, (u=2)  & 1& 2^{-1} & 2^{-1}\\
(\ge 8,10,14) & \kII^*&   & 1& 2& 1\\
(\ge 8,\ge 11,\ge 16) &\SetCell[c=5]{c}   \oplus\, (u=2)  \\
\end{longtblr}

\noindent Specifications for the fifth column in the above table:
\begin{itemize}
\item {\bf $1^*$ or $2^*$:} In the case $\operatorname{sig}_{2}(E)= (\ge 6,6,6)$, the value $u_2({\mathcal E}^d)$ is given by
 $$
 u_2({\mathcal E}^d)=\left\{\begin{array}{l}
  1 \,\,\text{if $c_6/2^6\not\equiv d/2\!\pmod{4}$},\\
  2 \,\,\text{if $c_6/2^6 \equiv d/2\!\pmod{4}$}.
 \end {array}\right.
 $$
\item {\bf $4^*$ or $2^*$:}  In the case $\operatorname{sig}_{2}(E)= (6, 9,18+n)$, the value $u_2({\mathcal E}^d)$ is given by
 $$
 u_2({\mathcal E}^d)=\left\{\begin{array}{l}
  4 \,\,\text{if $c_6/2^9\not\equiv d/2\!\pmod{4}$},\\
  2 \,\,\text{if $c_6/2^9 \equiv d/2\!\pmod{4}$}.
 \end {array}\right.
 $$
\end{itemize}

\vskip 0.3truecm 

\noindent {\bf Faltings elliptic curves.} 
We shall make use of the Faltings height of any elliptic curve $E$ defined over $\Q$, given by:
$$\
h(E) = -\frac{1}{2} \log (
\operatorname{vol}(\Lambda(\Emin)))
$$
where again $\Lambda(\Emin)$ denotes the N\'eron lattice attached to $E$. Stevens shows in \cite[Theorem 2.3]{S}  that within each $\mathbb{Q}$-isogeny class, there exists a unique elliptic curve with minimal Faltings height. Recall that we refer to this curve as the Faltings elliptic curve in the $\Q$-isogeny class
$\operatorname{Isog}(E)$.

\section{Strategy} 
\label{strategy}
As already mentioned before, our aim is to discuss which vertex is the Faltings elliptic curve in every twisted 
$\Q$-isogeny graph. In what follows, we describe our strategy, which depends on the genus of the corresponding modular curve. 

\begin{longtblr}
[
label = {types0},
caption = {Types of $\Q$-isogeny graphs: $X_0(N)$ of genus $0$}
]
{cells = {mode=imath},hlines,vlines,measure=vbox,colspec  = cccccccccccccc}
N & 2 & 3 & 4 & 5 & 6 & 7 & 8 & 9 & 10 & 12 & 13 & 16 & 25 \\
\text{type} & L_2(2) & L_2(3) & T_4 & L_2(5) & R_4(6) & L_2(7) & T_6 & L_3(9) & R_4(10) & S_8 & L_2(13) & T_8 & L_3(25)
\end{longtblr}

\begin{longtblr}
[
label = {typesgt0},
caption = {Types of $\Q$-isogeny graphs: $X_0(N)$ of genus $>0$}
]
{cells = {mode=imath},hlines,vlines,measure=vbox,colspec  = cccccccccccccc}
N & 11 & 14 & 15& 19 & 21 & 27 & 37 & 43 & 67 & 163    \\
\text{type} & L_2(11) & R_4(14) & R_4(15) & L_2(19)  &  R_4(21) & L_4 &  L_2(37) &  L_2(43) &  L_2(67) &  L_2(163)   
\end{longtblr}

\vskip 0.3 truecm

\noindent {\bf Cases $X_0(N)$ have genus 0.} The function field 
$\Q(X_0(N))=\Q(t(\tau))$ is generated by a modular function $t(\tau)$ (hauptmodul).
We refer to \cite[Table 2]{alvaro} for the choice of such a modular function~$t(\tau)$.
Each non-cuspidal rational value $t=t(\tau)\in \Q$
gives rise to a $\Q$-isogeny graph (and its twisted graphs) of type associated with $X_0(N)$.

\vskip 0.2truecm

{\it Step 1:} We compute the $j$-invariants and set the signatures $\operatorname{sig}(E)=(c_4,c_6,\Delta)$ for every vertex~$E$ in the $\Q$-isogeny graph of
$\operatorname{Isog}
=\operatorname{Isog}(t)$
corresponding to the non-cuspidal rational value~$t$.
The elliptic curves $E=E(c_4,c_6,\Delta)$ are computed using Velú's formulas in such a way that the isogenies in the resulting graph are normalized. See ~\cite{githubrepo} for the precise data (also \cite{Barrios} and the references therein for similar
parametrizations).

\vskip 0.2truecm

{\it Step 2:} Depending on the $p$-adic valuations of the rational value $t$, for each vertex $E=E(c_4,c_6,\Delta)$ in the $\Q$-isogeny graph, we compute the value 
$u =u(E) = \prod_p u_p(E)$ to apply a Weierstrass transformation giving rise to the signature of a global minimal model $\mathcal E=\mathcal E(u^{-4}c_4,u^{-6}c_6,u^{-12}\Delta)$ of each elliptic curve $E$. This is accomplished through Tate's algorithm as described by Tables
\ref{PapaNo23}, \ref{Papa3}, and \ref{Papa2}  in  Section \ref{terminology}.
From this, we save the (projective) vector 
$${\bf u}=
[u(E)\colon E\in \operatorname{Isog}].$$

\vskip 0.2truecm

{\it Step 3:} Depending on the square-free integer $d$, we compute the Pal value
$u_d=u({\mathcal E}^d)=\prod_p u_p({\mathcal E}^d)$
to apply a Weierstrass transformation giving rise to a global minimal model of the quadratic twist~${\mathcal E}^d$ over $\Q(\sqrt{d})$. Again, this is accomplished by using  the Tables
\ref{PapaNo23}, \ref{Papa3}, and \ref{Papa2}  in  Section \ref{terminology}. At this point, the signature of a minimal model of ${\mathcal E}^d$ is
$
(u^{-4} u_d^{-4} d^2 c_4,u^{-6} u_d^{-6} d^3 c_6,u^{-12}u_d^{-12} d^6\Delta)$.
From this, we save the (projective) vector 
$${\bf u}_d=
[u(\Emind)\colon E\in \operatorname{Isog}].$$

{\it Step 4:} Finally, we can employ the above projective vectors
${\bf u}$, ${\bf u}_d$
and apply Lemma \ref{normal} to compare the Faltings heights for the
different vertices $E^d$ in the twisted $\Q$-isogeny graph. This allows us to decide which of them corresponds to the Faltings elliptic curve in $\operatorname{Isog}(E^d)$. Notice that the volumes of the corresponding Néron lattices of the quadratic twists $E^d$ satisfy:
$$
\operatorname{vol}
\left(
\frac{u(\Emind)}{\sqrt{d}} u(E) \Lambda_E\right)=
\frac{u(\Emind)^2}{|d|} u(E)^2 
\operatorname{vol}(\Lambda_E)\,.
$$

\vskip 0.2truecm

\noindent See Section \ref{examples} for the case associated with the modular curve $X_0(9)$ (graphs of type $L_3(9)$).

\vskip 0.3truecm

\noindent {\bf Cases $X_0(N)$ has genus $\geq 1$.}
Under the genus assumption, it turns out that the non-cuspidal points in the union of all $X_0(N)(\Q)$ are in finite number. We refer to \cite[Table 4]{alvaro} for the  $j$-invariants of the elliptic curves attached to all these points.
For every non-cuspidal rational point $P=(j(\tau),j(N\tau))$ in $X_0(N)(\Q)$ we chose minimal Weierstrass models for the vertices in the corresponding $\Q$-isogeny graph and detect the Faltings elliptic curve on it (for instance, using \cite{sage}). Then, for every square free integer $d$, we use again Pal's values to have control of the Faltings elliptic curve in the twisted isogeny graph over $\Q(\sqrt{d})$. 

\vskip 0.2truecm

\noindent See Section \ref{examples} for the case
associated with the modular curve  $X_0(11)$ (graphs of type $L_2(11))$.

\section{Examples}
\label{examples}
\noindent {\bf Case $L_3(9)$.}
The isogeny graphs of type $L_3(9)$ are given by
three isogenous elliptic curves:
\[ 
\begin{tikzcd}
E_1 \arrow[dash]{r}{3} & E_3  \arrow[dash]{r}{3} & E_9   \,.
\end{tikzcd}
\]
The non-cuspidal rational points of the modular curve $X_0(9)$ parametrize the isogeny graphs of type~$L_3(9)$. The modular curve $X_0(9)$ has genus $0$ and a hauptmodul for this curve is:
$$t(\tau)= 3^3 \left( \frac{\eta(9\tau)}{\eta(\tau)}\right)^3\,.$$ 
Letting $t=t(\tau)$, one can write
$$
\begin{tblr}{l@{\,=\,}l}
j(E_1) = j(\tau) & 
\displaystyle{\frac{(t + 3)^{3}(t^{3} + 9 \, t^{2} + 27 \, t + 3)^{3}}{t(t^{2} + 9 \, t + 27)}}\\[8pt]
j(E_3) = j(3\tau) & 
\displaystyle{\frac{(t + 3)^{3}(t + 9)^{3}}{t^3(t^{2} + 9\,  t + 27)^{3} }}
\\[8pt]
j(E_9) = j(9\tau) & 
\displaystyle{\frac{ (t + 9)^{3}(t^{3} + 243 \, t^{2} + 2187\,  t + 6561)^{3}}{t^9(t^{2} + 9\,  t + 27) }}\,.\\[8pt]
\end{tblr}
$$
By using Velú formulas, we can choose Weierstrass equations for $(E_1,E_3,E_{9})$ in such a way that the isogenies in the graph are normalized. Their signatures are given by:

\[
\begin{tblr}{|c|l|}
\hline \SetCell[c=2]{c} L_3(9) \\ \hline
c_4(E_1) & 
(t + 3)  (t^{3} + 9\,  t^{2} + 27\,  t + 3)\\
c_6(E_1) & 
t^{6} + 18\,  t^{5} + 135\,  t^{4} + 504 \, t^{3} + 891\,  t^{2} + 486\,  t - 27\\  
\Delta(E_1) & 
t  (t^{2} + 9\,  t + 27)\\
\hline
c_4(E_3) & 
(t + 3)  (t + 9) (t^{2} + 27) \\
c_6(E_3) & 
(t^{2} - 27) (t^{4} + 18\,  t^{3} + 162\,  t^{2} + 486\,  t + 729) \\  
\Delta(E_3) & 
t^{3}  (t^{2} + 9 \, t + 27)^{3} \\
\hline
c_4(E_9) & 
(t + 9)  (t^{3} + 243\,  t^{2} + 2187\,  t + 6561) \\
c_6(E_9) & 
t^{6} - 486\,  t^{5} - 24057\,  t^{4} - 367416\,  t^{3} - 2657205\,  t^{2} - 9565938\,  t - 14348907\\  
\Delta(E_9) & 
 t^{9}  (t^{2} + 9\,  t + 27)\\
\hline
\end{tblr}
\]

\vskip 0.3truecm

We remark that the subgroup of $\operatorname{Aut} X_0(9)$ that fixes the set of vertices of the graph is
generated by the Fricke involution of $X_0(9)$, given by $W_9(t)= 3^3/t$. The involution $W_9$ acts on the
isogeny graphs of type $L_3(9)$ as:
$
\begin{tikzcd}
W_9(E_1 \arrow[dash]{r}{3} & E_3\arrow[dash]{r}{3} & E_9   )=E_9^{-3} \arrow[dash]{r}{3} & E_3^{-3}  \arrow[dash]{r}{3} & E_1^{-3}  \,.
\end{tikzcd}
$
This information can be used to derive symmetries in the tables displayed below. 
With regard to Kodaira symbols, minimal models, and Pal values, using the above terminology, we get:

\begin{longtblr}
[label={L39_pNo23}, caption= $L_3(9)$ data for $p\neq 2${,} $3$]
{cells={mode=imath},hlines,vlines,measure=vbox,
colspec=cclclccc}
L_3(9) & \SetCell[c=5]{c} p\ne 2,3 & & & & \\
t & E & \SetCell[c=1]{c} \operatorname{sig}_p(\Emin) & u_p(E) & \Kd_p(E) & u_p(\Emind) \\
\SetCell[r=3]{c} v_p(t)=m> 0 
& E_1 & (0,0,m) & 1 & \kI_m & 1\\
& E_3 & (0,0,3m) & 1 & \kI_{3m}& 1 \\
& E_9 & (0,0,9m) & 1 & \kI_{9m} & 1\\
\SetCell[r=3]{c} 
\begin{array}{c}
v_p(t)=0  \\[6pt]
m=v_p(t^2+9\,t+27) \geq 0 
\end{array}
& E_1 & (0,0,m) & 1 & \kI_{m} & 1 \\
& E_3 & (0,0,3m) & 1 & \kI_{3m} & 1 \\
& E_9 & (0,0,m) & 1 & \kI_{m} & 1 \\
\SetCell[r=3]{c} v_p(t)=-m < 0 
& E_1 & (0,0,9m) & p^{-m} & \kI_{9m} & 1  \\
& E_3 & (0,0,3m) & p^{-m} & \kI_{3m} & 1 \\
& E_9 & (0,0,m) & p^{-m} & \kI_{m} & 1 \\
\end{longtblr}

\begin{longtblr}
[label={L39_p3},caption= {$L_3(9)$ data for $p$=3}]
{cells={mode=imath},hlines,vlines,measure=vbox,
hline{Z}={1-X}{0pt},
vline{1}={Y-Z}{0pt},
colspec=cclclcc,rowhead = 2}
L_3(9) &\SetCell[c=6]{c} p=3  & & & & & \\
t & E & \SetCell[c=1]{c} \operatorname{sig}_3(\Emin) & u_3(E) & \Kd_3(E) & \SetCell[c=2]{c} u_3(\Emind)  & \\
\SetCell[r=3]{c} v_3(t)=m\ge 3 
& E_1 & (2,3,m+3) & 1 & {\kI}^*_{m-3} & 3 & 1\\
& E_3 & (2,3,3m-3) & 3 & {\kI}^*_{3(m-3)} & 3 & 1 \\
& E_9 & (2,3,9m-21) & 3^{2} & {\kI}^*_{9(m-3)} & 3 & 1 \\
\SetCell[r=3]{c} 
     v_3(t)=2   
& E_1 & (2,3,5) & 1 & \kIV  & 1 & 1\\
& E_3 & (\ge 2,3,3) & 3 & \kII  & 1 & 1\\
& E_9 & (\ge 4,6,9) & 3 & {\kIV}^* & 3 & 1 \\
\SetCell[r=3]{c} 
     v_3(t)=1  
& E_1 & (\ge 2,3,3) & 1 & \kII  & 1 & 1\\
& E_3 & (\ge 4,6,9) & 1 & {\kIV}^* & 3 & 1 \\
& E_9 & (4,6,11) & 1 & {\kII}^*  & 3 & 1\\
\SetCell[r=3]{c} 
    v_3(t)=-m\le 0   
& E_1 & (0,0,9m) & 3^{-m} & \kI_{9m}  & 1 & 1 \\
& E_3 & (0,0,3m) & 3^{-m} & \kI_{3m}  & 1 & 1 \\
& E_9 & (0,0,m) & 3^{-m} & \kI_{m}  & 1 & 1 \\
 \SetCell[c=5,r=2]{c} & & & & & d\equiv 0  & d\not\equiv 0 \\
                      & & & & & \SetCell[c=2]{c} d \Mod 3 & \\
\end{longtblr}



\begin{longtblr}
[label={L39_p2},caption = {$L_3(9)$ data for $p$=2}]
{cells = {mode=imath},hlines,vlines,measure=vbox,
hline{Z} = {1-5}{0pt},
vline{1} = {Y-Z}{0pt},
colspec  = cclclccc}
L_3(9) & \SetCell[c=7]{c} p=2  & & & & & \\ 
t & E & \SetCell[c=1]{c} \operatorname{sig}_2(\Emin) & u_2(E) & \SetCell[c=1]{c} \Kd_2(E) & \SetCell[c=3]{c} u_2(\Emind)  \\
\SetCell[r=3]{c} v_2(t)=m>0 
& E_1 & (4,6,m+12) & 2^{-1} & \kI_{m+4}^* & 1 & 1 & 2 \\
& E_3 & (4,6,3m+12) & 2^{-1} & \kI_{3m+4}^* & 1 & 1 & 2\\
& E_9 & (4,6,9m+12) & 2^{-1} & \kI_{9m+4}^*& 1 & 1 & 2\\
\SetCell[r=3]{c} v_2(t)=0 
& E_1 & (\geq 8,9,12) & 2^{-1} & \kII^* & 1 & 2 & 2 \\
& E_3 & (\geq 8,9,12) & 2^{-1} & \kII^* & 1 & 2 & 2\\
& E_9 & (\geq 8,9,12) & 2^{-1} & \kII^*& 1 & 2 & 2\\
\SetCell[r=3]{c} v_2(t)=-m<0 
& E_1 & (4,6,9m+12) & 2^{-m-1} & \kI_{9m+4}^* & 1 & 1 & 2\\
& E_3 & (4,6,3m+12) & 2^{-m-1}  & \kI_{3m+4}^*& 1 & 1 & 2\\
& E_9 & (4,6,m+12) & 2^{-m-1}  & \kI_{m+4}^*& 1 & 1 & 2 \\
 \SetCell[c=5,r=2]{c} & & & & &  d\equiv 1 &  d\equiv 2  & d\equiv 3 \\
                      & & & & & \SetCell[c=3]{c} d \Mod{4} & \\
\end{longtblr}

From the previous tables one gets the (projective) vectors 
${\bf u}=[u(E)]$ and ${\bf u}(d)=[u(\Emind)]$ as discussed in Section~\ref{strategy}:

\begin{longtblr}
[label={L39},caption = {$L_3(9)$ data}]{cells={mode=imath},hlines,vlines,measure=vbox, rowhead = 1}
\SetCell[c=1]{c} t &\SetCell[c=1]{c} [u(E)]  & \SetCell[c=1]{c} [u(\Emind)] & \SetCell[c=1]{c} d \\
\SetCell[r=1]{c} v_3(t)\le 0 & \SetCell[r=1]{c} (1:1:1) & (1:1:1) &\SetCell[c=1]{c}  \text{all} \\
\SetCell[r=2]{c} v_3(t)=1 & \SetCell[r=2]{c} (1:1:1) & (1:1:1) & d\not\equiv 0\,(3) \\
& & \SetCell[r=1]{c} (1:3:3) & d\equiv 0\,(3) \\
\SetCell[r=2]{c}  v_3(t)=2 & \SetCell[r=2]{c} (1:3:3) & (1:1:1) & d\not\equiv 0\,(3) \\
& & \SetCell[r=1]{c} (1:1:3) & d\equiv 0\,(3) \\
v_3(t)\ge 3 & (1:3:3^{2}) & (1:1:1) &\SetCell[c=1]{c}  \text{all} \\
\end{longtblr}

\vskip 0.2truecm
\noindent The contents of the above table are the ingredients to prove the following result:

\begin{proposition}
Let 
$ 
\!\!\begin{tikzcd}
E_1 \arrow[dash]{r}{3}  & E_3 \arrow[dash]{r}{3} & E_9 
\end{tikzcd}\!\!
$
be a $\mathbb{Q}$-isogeny graph of type $L_3(9)$ corresponding to a given $t$ in $\mathbb{Q}^*$ with signatures as above. For every square-free integer $d$, 
the Faltings curve (circled)
in the twisted isogeny graph 
$
\!\!\begin{tikzcd} 
E_1^d \arrow[dash]{r}{3}  & E_3^d\arrow[dash]{r}{3} & E_9^d 
\end{tikzcd}\!\!
$ 
is given by:

\[
\begin{tblr}{|c|c|c|c|c|}
\hline
 L_3(9) & \text{twisted isogeny graph} & d & \text{Prob} \\
\hline
 \SetCell[r=1]{c} v_3(t)\leq 0 & \circled[0.8]{$E_1^d$} \longrightarrow E_3^d\longrightarrow E_9^d  & \text{all}& 1 \\
\hline
\SetCell[r=2]{c} v_3(t)=1   
& \circled[0.8]{$E_1^d$} \longrightarrow E_3^d\longrightarrow E_9^d & d\not\equiv 0\,(3) & 3/4\\ 
&  E_1^d \longleftarrow \circled[0.8]{$E_3^d$} \longrightarrow E_9^d &d\equiv 0\,(3) &  1/4 \\
\hline
\SetCell[r=2]{c}
v_3(t)=2   
 &  E_1^d \longleftarrow \circled[0.8]{$E_3^d$} \longrightarrow E_9^d &d\not\equiv 0\,(3) & 3/4 \\
 &  E_1^d \longleftarrow E_3^d\longleftarrow \circled[0.8]{$E_9^d$} &d\equiv 0\,(3) & 1/4 \\
\hline
 \SetCell[r=1]{c} v_3(t)\ge 3   &  E_1^d \longleftarrow  E_3^d\longleftarrow  \circled[0.8]{$E_9^d$} & \text{all} & 1 \\
\hline
\end{tblr}
\]
\vskip 0.1truecm
\noindent The column Prob gives the probability of the circled curve to be the Faltings curve.
\end{proposition}

\begin{proof}
We begin by showing how the data shown  in  Tables \ref{L39_p3}, \ref{L39_pNo23}, and \ref{L39_p2} is obtained. To illustrate the process, we consider the case where \( v_3(t) = 2 \) in Table \ref{L39_p3}. The remaining cases can be handled in a similar manner.  
Writing \( t = 3^2 a \) where  \( a=a_0+O(3) \in \mathbb{Z}_3 \) is a unit, 
we can express the signature 
\(\operatorname{sig}(E_1)\) as  
\[
\begin{array}{l}
c_4(E_1) = 3^2 (3 a+1) \left(3^5 a^3+3^5 a^2+3^4 a+1\right)\\[3pt]
c_6(E_1) = 3^3 \left(3^9 a^6+2\cdot 3^9 a^5+3^8\cdot 5 a^4+2^3\cdot 3^5 \cdot 7 a^3+3^5\cdot 11 a^2+2 \cdot 3^4 a-1\right)\\[3pt]
\Delta(E_1) = 3^5 a \left(3 a^2+3 a+1\right).
\end{array}
\]
We see that $\operatorname{sig}_3(E_1) = (2,3,5)$. By using Table \ref{Papa3} (Section \ref{terminology}), it follows that \( E_1 \) is minimal, i.e., \( u_3(E_1) = 1 \), and its Kodaira symbol is $\Kd_3(E_1) = \kIV$. Additionally, the Pal value in this case is $
u_3(\mathcal{E}_1^d) = 1$ for any $d$. Similarly, for the signature of the elliptic curve \( E_3 \), we obtain  
\[
\begin{array}{l}
c_4(E_3) = 3^6 (a+1) (3 a+1) \left(3 a^2+1\right)\\[3pt]
c_6(E_3) = 3^9 \left(3 a^2-1\right) \left(3^2 a^4+2\cdot 3^2 a^3+18 a^2+2\cdot 3a+1\right)\\[3pt]
\Delta(E_3) = 3^{15} a^3 \left(3 a^2+3 a+1\right)^3.
\end{array}
\]  
Hence  $\operatorname{sig}_3(E_3) = (\geq 6,9,15)$. Now Table \ref{Papa3} indicates that the Weierstrass model used for \( E_3 \) with this signature is not minimal. Performing a transformation with \( u_3(E_3) = 3 \) we obtain a Weierstrass model with $
\operatorname{sig}_3(E_3) = (\geq 2,3,3)$ which gives a $3$-minimal model. Moreover, to determine the Kodaira symbol we must check condition $\cond{3_a}$:  
\[
\left(\frac{c_6}{3^3}\right)^2+2-3\left(\frac{c_4}{3^2}\right) \equiv 3 a_0^2 \equiv 3 \!\pmod{9}.
\]  
Therefore, we conclude that $\Kd_3(E_3) = \kII$  and $u_3(\mathcal{E}_3^d) = 1$ for any square-free integer $d$.

Finally, for the elliptic curve \( E_9 \), we obtain  
\[
\begin{array}{l}
c_4(E_9) = 3^8 (a+1) \left(a^3+3^3 a^2+3^3 a+3^2\right) \\[3pt]
c_6(E_9) = 3^{12} \left(a^6 - 2\cdot 3^3 a^5 - 3^3\cdot 11 a^4 - 2^3\cdot 3^2\cdot 7 a^3 - 3^4\cdot 5 a^2 - 2\cdot 3^4 a - 3^3\right) \\[3pt]
\Delta(E_9) = 3^{21} a^9 \left(3 a^2+3 a+1\right).
\end{array}
\]  
Now, one has $\operatorname{sig}_3(E_9) = (\geq 8,12,21)$ so that we need to apply a transformation with \( u_3(E_9) = 3 \) to get a $3$-minimal model with $
\operatorname{sig}_3(E_9) = (\geq 4,6,9)$. In this case, to determine the Kodaira symbol, we must check whether condition 
$\cond{3_b}$ holds:  
\[
\left(\frac{c_6}{3^6}\right)^2+2-3\left(\frac{c_4}{3^4}\right)  \equiv a_0^{12}+6 a_0^4+6 a_0^3+2 \equiv \pm 3 \!\pmod{9}.
\]  
Hence, it follows  $\Kd_3(E_9) = {\kIV}^*$. Furthermore, we get $
u_3(\mathcal{E}_9^d) = 3$ if $d \equiv 0 \!\pmod{3}$ whereas $u_3(\mathcal{E}_9^d) = 1$ if $d \not\equiv 0 \!\pmod{3}$.

For primes $p\neq 3$, a similar discussion yields to the information stored in Tables  \ref{L39_p3}, \ref{L39_pNo23}, and \ref{L39_p2}. Indeed, for 
\( p \ne 3 \), it holds 
\[
(u_p(E_1) : u_p(E_3) : u_p(E_9)) = (u_p(\Emind_1) : u_p(\Emind_3) : u_p(\Emind_9)) = (1 : 1 : 1),
\]
and therefore we get the projective vectors
\[
[u(E)] = (u_3(E_1) : u_3(E_3) : u_3(E_9)) \quad\text{and}\quad
[u(\Emind)] = (u_3(\Emind_1) : u_3(\Emind_3) : u_3(\Emind_9)).
\]
Summarizing, for the case $v_3(t)=2$, we have $[u(E)] = (1:3:3)$, and $[u(\Emind)] = (1:1:1)$ if $d \not\equiv 0 \!\pmod{3}$ or $[u(\Emind)] = (1:1:3)$ if $d \equiv 0 \!\pmod{3}$.

Since the signatures are taken in such a way that the isogenies are normalized, we can write the period lattices of $(E_1,E_3,E_9)$ as:
$$
\displaystyle{
\Lambda_1 = \lambda \langle 1, \tau \rangle \,,\qquad
\Lambda_3 = \frac{1}{3}\lambda \langle 1, 3\tau \rangle \,,\qquad
\Lambda_9 = \frac{1}{9} \lambda \langle 1, 9\tau \rangle \,,}   
$$
for some $\lambda,\tau\in \mathbb{C}$ and $\operatorname{Im}(\tau)>0$. Letting $v=\operatorname{vol}(E_1)$, projectively we obtain
$$
(\operatorname{vol}(E_1):\operatorname{vol}(E_3):\operatorname{vol}(E_9))=\left(v:\frac{v}{3}:\frac{v}{9}\right)=\left(1:\frac{1}{3}:\frac{1}{9}\right).
$$
By using that \( [u(E)] = (1:3:3) \), the volumes of the Néron lattices of \( (E_1, E_3, E_9) \) satisfy
$$
(\operatorname{vol}(\Emin_1):\operatorname{vol}(\Emin_3):\operatorname{vol}(\Emin_9))=\left(v:3^2\frac{v}{3}:3^2\frac{v}{9}\right)=\left(1:3:1\right).
$$
Finally, let us prove the case $v_3(t)=2$ and $d\equiv 0 \!\pmod{3}$. Then $[u(\Emind)] = (1:1:3)$ and the volumes of the N\'eron lattices of the quadratic twists $(\Emin_1^d,\Emin_3^d,\Emin_9^d)$ satisfy:
$$
\left(v:3^2\frac{v}{3}:3^2 3^2\frac{v}{9}\right)=\left(1:3:9\right).
$$
Hence, in these cases we conclude that $E^d_9$ has minimum Faltings height among the elliptic curves
in the $\Q$-isogeny class. By Lemma~\ref{analyticNT}, 
it turns out that the quadratic twists of $E_9$ that become the Faltings elliptic 
curve in their $\Q$-isogeny class are in a proportion of
$1$ out of $4$.
\end{proof}

\vskip 0.3truecm

\noindent {\bf Case $L_2(11)$.}
The isogeny graphs of type $L_2(11)$ are given by two $11$-isogenous elliptic curves:
\[ 
\begin{tikzcd}
E_1 \arrow[dash]{r}{11} & E_{11}\,.
\end{tikzcd}
\]
The non-cuspidal rational points of the modular curve $X_0(11)$ parametrize the isogeny graphs of type~$L_2(11)$. The modular curve $X_0(11)$ is an elliptic curve of rank $0$ over the rationals. More precisely, we can choose the Weierstrass model $y^2 + y = x^3 - x^2 - 10\, x - 20$ for $X_0(11)$. 
The $j$-forgetful map $j\colon X_0(11) \to X_0(1)$ is given by
$$
j=
\frac{P_2(y)x^2+P_1(y)x+P_0(y)}
   {\left(x-16\right)}
$$
with
$$
\begin{array}{l}
P_2(y)=-11 y^4+641 y^3-452 y^2+11803 y-14372,\\
P_1(y)=-24 y^3+536 y^2+4540 y+6341\\
P_0(y)=-y^3+125 y^2-502 y+1776\\
\end{array}
$$
We refer to \cite{X011} for the above computation. One has
$$
X_0(11)(\Q)=
\{
(0 : 1 : 0), (5 : -6 : 1), (5 : 5 : 1), (16 : -61 : 1), (16 : 60 : 1)\}
$$
and hence:
$$
j((0 : 1 : 0))= \infty\,,\quad
j((16 : -61 : 1))= \infty\,,
$$
$$
j((5 : 5 : 1))= - 2^{15} \,,\quad
j((5 : -6 : 1))= -11^2 \,,\quad
j((16 : 60 : 1))= - 11 \cdot 131^3 \,.
$$
Besides the cusps
$(\infty)=(0:1:0)$, $(0)=(16:-61:1)$, we have a rational CM point $(5 : 5 : 1)$ that corresponds to $\tau_b=\frac{1}{2}+\frac{\sqrt{-11}}{2\cdot 11}\in\mathbb H$, and two non-cuspidal non-CM points 
$(16 : 60 : 1)$ and 
$(5 : -6 : 1)$ that correspond to
$\tau_a=0.5+0.09227...i$, and $\tau'_a=0.5+0.24630...i\in\mathbb H$. We have
$$
j(\tau_b)=j(11\tau_b)=-2^{15},\qquad j(\tau_a)=-11^{2},\qquad 
j(\tau'_a)=j(11\tau_a)=-11\cdot 131^{3}\,.
$$
We choose minimal Weierstrass equations and get two normalized $\Q$-isogeny graphs:
\vskip 0.5truecm
\begin{tblr}
{cells={mode=imath},hlines,vlines,measure=vbox,
colspec=clll}
E & \text{Minimal Weierstrass model} & j(E) &\text{LMFDB}\\
\hline
E_{1_a} & y^2+xy+y=x^3+x^2-30x-76 & -11\cdot 131^{3} & \eclabel{121.a2}\\
E_{11_a} & y^2+xy+y=x^3+x^2-305x+7888 & -11^{2} & \eclabel{121.a1}
\end{tblr}

\vskip 0.3truecm

\begin{tblr}
{cells={mode=imath},hlines,vlines,measure=vbox,
colspec=clll}
E & \text{Minimal Weierstrass model} & j(E) &\text{label}\\
\hline
E_{1_b} & y^2+y=x^3-x^2-7x+10 & -2^{15} & \eclabel{121.b2}\\
E_{11_b} & y^2+y=x^3-x^2-887x-10143 & -2^{15} & \eclabel{121.b1}\\
\end{tblr}

\vskip 0.5truecm
\noindent Their signatures are:
\vskip 0.5truecm

\begin{tblr}{cc}
\begin{tblr}
{cells={mode=imath},hlines,vlines,measure=vbox,
colspec=cll}
 E  & E_{1_a}& E_{11_a}\\
c_4(E) & 11\cdot 131 & 11^{4}\\
c_6(E) & 11\cdot 4973 & -11^{5}\cdot 43\\  
\Delta(E) & -11^{2} & -11^{10}\\
\end{tblr}
&
\begin{tblr}
{cells={mode=imath},hlines,vlines,measure=vbox,
colspec=cll}
 E  & E_{1_b} &  E_{11_b} \\
\hline 
c_4(E) & 2^{5}\cdot 11
 & 2^{5}\cdot 11^{3}\\
c_6(E) & -2^{3}\cdot 7\cdot 11^{2}
 & 2^{3}\cdot 7\cdot 11^{5}  \\  
\Delta(E) & -11^{3}
 & -11^{9} \\
\end{tblr}
\end{tblr}


\vskip 0.5truecm

\noindent With the help of \verb|SageMath| or \verb|Magma|, one checks  that the Faltings curve (circled)
in the graphs are
\[
\begin{tblr}{|c|c|}
\hline
 \circled[0.8]{$E_{1_a}$} \longrightarrow E_{11_a} \\
\hline
\end{tblr}
\qquad
\begin{tblr}{|c|c|}
\hline
 \circled[0.8]{$E_{1_b}$} \longrightarrow E_{11_b} \\
\hline
\end{tblr}
\]
Note that any $\mathbb Q$-isogeny class of elliptic curves of type $L_2(11)$ is obtained by quadratic twist
from these two graphs.
With regard to Kodaira symbols, minimal models, and Pal values, observe that the unique prime of bad reduction for the above elliptic curves is $p=11$. One has: 

\vskip 0.3truecm

\begin{longtblr}
[label={L211_p11},caption = {$L_2(11)$ data for $p$=11}]
{cells = {mode=imath},hlines,vlines,measure=vbox,
hline{Z} = {1-3}{0pt},
vline{1} = {Y-Z}{0pt},
colspec  = cllcc,rowhead = 2}
\SetCell[c=5]{c} p=11   & & &  \\ 
 E & \SetCell[c=1]{c} \operatorname{sig}_{11}(\Emin)  & \SetCell[c=1]{c} \Kd_{11}(E) & \SetCell[c=2]{c} u_{11}(\Emind)  \\
\hline
 E_{1_a} & (1,1,2)  & \kII & 1 & 1  \\
 E_{11_a} &(4,5,10)   & \kII^* & 11 & 1 \\
\hline
 E_{1_b} & (1,2,3)   & \kIII& 1 & 1\\
 E_{11_b} & (3,5,9)   & \kIII^* & 11 & 1  \\
 \SetCell[c=3,r=2]{c}   & & &  d\equiv 0 &  d\not\equiv 0  \\
                        & & & \SetCell[c=2]{c} d \Mod{11}   \\
\end{longtblr}

From the above tables one gets the (projective) vector ${\bf u}(d)=[u(\Emind)]$:
\begin{longtblr}
[label={L211},caption = {$L_2(11)$ data}]{cells={mode=imath},hlines,vlines,measure=vbox}
 \SetCell[c=1]{c} [u(\Emind)] & \SetCell[c=1]{c} d \\
 \SetCell[r=1]{c} (1:1) & d\not\equiv 0\,(11) \\
  \SetCell[r=1]{c} (1:11) & d\equiv 0\,(11) \\
\end{longtblr}

Using the above data, we obtain the following result:

\begin{proposition}
For every square-free integer $d$, the Faltings curve (circled)
in the twisted isogeny graph 
$
\!\!\begin{tikzcd} 
E_{1_k}^d \arrow[dash]{r}{11}  & E_{11_k}^d
\end{tikzcd}\!\!
$, for $k\in\{a,b\}$, is given by:

\[
\begin{tblr}{|c|c|c|}
\hline
\text{twisted isogeny graph} & \text{condition}  &\text{Prob} \\
\hline
 \circled[0.8]{$E_{1_k}^d$} \longrightarrow E_{11_k}^d  & d\not\equiv 0\,(11) & 11/12 \\
\hline
 E_{1_k}^d \longrightarrow \circled[0.8]{$E_{11_k}^d$}  &  d\equiv 0\,(11) & 1/12 \\
\hline
\end{tblr}
\]
\vskip 0.2truecm
\noindent
The column Prob gives the probability of the circled curve to be the Faltings curve.
\end{proposition}

\begin{proof}
Let $k \in \{a, b\}$, and for simplicity, denote $E_1 = E_{1_k}$ and $E_{11} = E_{11_k}$. The chosen models are already minimal, so we only need to determine the corresponding Pal value for each square-free integer $d$ using Table \ref{pal}. For all elliptic curves in the two isogeny classes, the only prime of bad reduction is $11$. Thus, for any $d$, we have $u_p(\mathcal E_{1}^d) = u_p(\mathcal E_{11}^d) = 1$, for $p \neq 2,11$. Now, consider the case $p=2$. Since $\operatorname{sig}_2(E_{1}) = \operatorname{sig}_2( E_{11})$, it follows that $u_2(\mathcal E_{1}^d) = u_2(\mathcal E_{11}^d)$. Finally, for $p=11$, analyzing the reduction modulo $11$ yields Table \ref{L211_p11}. This completes the proof of Table \ref{L211}.

Next, we have chosen the Weierstrass models such that the isogenies are normalized. If $\Lambda_1$ and $\Lambda_{11}$ denote the lattices associated with $E_1$ and $E_{11}$, respectively, then we have $\Lambda_1 = \lambda \langle 1, \tau \rangle$ and $\Lambda_{11} = \frac{1}{11}\lambda \langle 1, 11\tau \rangle$, for some $\lambda, \tau \in \mathbb{C}$ with $\operatorname{Im}(\tau) > 0$. Consequently, we obtain $\operatorname{vol}(E_{11}) = \frac{1}{11} \operatorname{vol}(E_1)$.

Finally, using the vector $[u(\Emin^d)]$ from Table \ref{L211}, we conclude that the volumes of the N\'eron lattices of $(\Emin_1^d, \Emin_{11}^d)$ satisfy:
\[
    \left\{
    \begin{array}{cll}
        \left(1:\frac{1}{11}\right) & \text{if } d \not \equiv 0 \pmod{11}, \\[2mm]
        \left(1:11 \right) & \text{if } d \equiv 0 \pmod{11}.
    \end{array}
    \right.
\]
This completes the proof.
\end{proof}

\section{Main result}\label{sec_MainTheorem}

Our main result is to establish the Faltings elliptic curve in every $\Q$-isogeny class of elliptic curves defined over $\Q$. Due to space constraints, the following statement requires a reference to the github repository \cite{githubrepo}.
For the $\Q$-isogeny graphs
arising from modular curves $X_0(N)$ of genus~0, the signatures
of the elliptic curves
in every $\Q$-isogeny class $\operatorname{Isog}$ depend on a rational parameter $t$ and they can be read from \cite{githubrepo}.
Similarly for the finite number of cases arising from modular curves 
$X_0(N)$ of genus $\geq 1$.

\begin{theorem}
\label{MainTheorem}
The following table gives the Faltings elliptic curve in every twisted $\Q$-isogeny class of elliptic curves defined over $\Q$.
The first column displays the type of the $\Q$-isogeny graph of
$\operatorname{Isog} = (E_{i_1}, \dots, E_{i_n})$, where $I=\{i_1,\dots,i_n\}$ (see Table \ref{first} for the definition of $I$ for each type). The second column is linked to the signatures of the elliptic curves $E_i$
as described in~\cite{githubrepo}. The third column displays the Faltings elliptic curve $E_i^d$ in the twisted graph
according to the conditions on the square-free integer $d$ (fourth column) and with probability given in the last column.
\end{theorem}

\begin{longtblr}
[label={table_theorem}, caption = {Faltings elliptic curve in twisted 
$\mathbb Q$-isogeny classes}]
{cells = {mode=imath},
hlines,
,vlines,
measure=vbox,
colspec  = cccccc,
hline{1,2,12,22,28,32,34,36,38,40,41,43,45,47,53,55,57,65,77,79, 81,83,87,96,100,104,110}={1.2pt,solid},
rowhead=1
}
 \text{type} &\SetCell[c=2]{c}  t & & \text{Faltings} & d & \text{Prob} \\
 \SetCell[r=10]{c} L_2(2)&  \SetCell[c=2]{c} v_2(t)\ge 8  & &  E_2^d &\text{all} & 1 \\
& \SetCell[r=2,c=2]{c}  v_2(t)=7   & &  E_2^d&d\not\equiv 0\,(2) & 2/3 \\
&   & & E_1^d &d\equiv 0\,(2) & 1/3 \\
& \SetCell[r=2,c=2]{c}  \begin{array}{c}
 v_2(t)=6\\
v_2(t+64)\equiv 2,3(4)
 \end{array}   & &  E_2^d&d\not\equiv 0\,(2) & 2/3 \\
&   & & E_1^d &d\equiv 0\,(2) & 1/3 \\
& \SetCell[r=2,c=2]{c} 
 \begin{array}{c}
 v_2(t)=6\\
v_2(t+64)\equiv 0,1(4)
 \end{array}&
 & E_1^d  &d\not\equiv 0\,(2) & 2/3 \\
&    & &  E_2^d  &d\equiv 0\,(2) & 1/3 \\
& \SetCell[r=2,c=2]{c} 
v_2(t)=5 &
 & E_1^d  &d\not\equiv 0\,(2) & 2/3 \\
&    & &  E_2^d  &d\equiv 0\,(2) & 1/3 \\
&  \SetCell[c=2]{c} v_2(t)\leq 4 & &E_1^d & \text{all} & 1 \\

\SetCell[r=10]{c} L_2(3) & \SetCell[c=2]{c} v_3(t)\ge 5  &  &  E_3^d & \text{all}& 1 \\
& \SetCell[r=2,c=2]{c} 
 v_3(t)=4    &  & E_3^d &d\not\equiv 0\,(3) & 3/4 \\
& & &  E_1^d &d\equiv 0\,(3) & 1/4 \\
& \SetCell[r=2, c=2]{c} 
\begin{array}{c}
 v_3(t)=3\\
 v_3(t+27)\equiv 3,4,5\,(6)
 \end{array}  
 &  & E_3^d &d\not\equiv 0\,(3) & 3/4 \\
& & &  E_1^d &d\equiv 0\,(3) & 1/4 \\
& \SetCell[r=2,c=2]{c} 
 \begin{array}{c}
 v_3(t)=3\\
 v_3(t+27)\equiv 0,1,2\,(6)
 \end{array} &
 & E_1^d   &d\not\equiv 0\,(3) & 3/4 \\
&   &  &  E_3^d &d\equiv 0\,(3) & 1/4 \\
& \SetCell[r=2,c=2]{c} 
 v_3(t)=2  &
 & E_1^d   &d\not\equiv 0\,(3) & 3/4 \\
&   &  &  E_3^d &d\equiv 0\,(3) & 1/4 \\
&  \SetCell[c=2]{c} v_3(t)\leq 1& &E_1^d & \text{all}& 1 \\
\SetCell[r=6]{c} L_2(5) &\SetCell[c=2]{c}  v_5(t)\ge 3 & & E_5^d &\text{all} & 1 \\
 & \SetCell[c=2,r=2]{c} v_5(t)=2  & &E_5^d & d\not\equiv 0\,(5) & 5/6\\ 
& &  &  E_1^d &d\equiv 0\,(5) &  1/6 \\
 &\SetCell[c=2,r=2]{c} v_5(t)=1&  & E_1^d  & d\not\equiv 0\,(5) & 5/6\\ 
& &&   E_5^d &d\equiv 0\,(5) &  1/6 \\
& \SetCell[c=2]{c} v_5(t)\le 0 & &E_1^d & \text{all} & 1 \\
\pagebreak
  \SetCell[r=4]{c}  L_2(7) &  \SetCell[c=2]{c} v_7(t)\ge 2 & & E_7^d & \text{all} & 1 \\
& \SetCell[c=2,r=2]{c} v_7(t)=1   &
&E_1^d  & d\not\equiv 0\,(7) & 7/8\\ 
& &  &  E_7^d &d\equiv 0\,(7) &  1/8 \\
&  \SetCell[c=2]{c} v_7(t)\le 0 & & E_1^d &\text{all}  & 1 \\
  \SetCell[r=2]{c}  L_2(11) &   \SetCell[c=2,r=2]{c}   \text{genus $\geq 1$}  &  & E_{1}^d & d\not\equiv 0\,(11) & 11/12 \\
&  & &  E_{11}^d  &  d\equiv 0\,(11) & 1/12 \\
  \SetCell[r=2]{c}  L_2(13) &  \SetCell[c=2]{c} v_{13}(t)>0 &\text{all}& E_{13}^d  & \SetCell[r=2]{c}  \text{all} &\SetCell[r=2]{c} 1 \\
 &\SetCell[c=2]{c}  v_{13}(t)\le 0 &\text{all} & E_1^d & &  \\
  \SetCell[r=2]{c}  L_2(17) &   \SetCell[c=2,r=2]{c}   \text{genus $\geq 1$}  & & E_{1}^d & d\not\equiv 0\,(17) & 17/18 \\
& &  &E_{17}^d  &  d\equiv 0\,(17) & 1/18 \\
  \SetCell[r=2]{c}  L_2(19) &   \SetCell[c=2,r=2]{c}   \text{genus $\geq 1$}   & & E_{1}^d & d\not\equiv 0\,(19) & 19/20 \\
& & & E_{19}^d  &  d\equiv 0\,(19) & 1/20 \\
  \SetCell[r=1]{c}  L_2(37) &  
  \SetCell[c=2, r=1]{c} \text{genus $\geq 1$}  &  & E_{1}^d &  \text{all} & 1 \\
  \SetCell[r=2]{c}  L_2(43) &  \SetCell[c=2,r=2]{c}   \text{genus $\geq 1$}  & & E_{1}^d & d\not\equiv 0\,(43) & 43/44 \\
& & & E_{43}^d &  d\equiv 0\,(43) & 1/44 \\
  \SetCell[r=2]{c}  L_2(67) &   \SetCell[c=2,r=2]{c}   \text{genus $\geq 1$}   & & E_{1}^d & d\not\equiv 0\,(67) & 67/68 \\
& & &  E_{67}^d &  d\equiv 0\,(67) & 1/68 \\
  \SetCell[r=2]{c}  L_2(163) &   \SetCell[c=2,r=2]{c}   \text{genus $\geq 1$}   & & E_{1}^d & d\not\equiv 0\,(163) & 163/164 \\
& & &  E_{163}^d  &  d\equiv 0\,(163) & 1/164 \\
 \SetCell[r=6]{c}  L_3(9)  &   \SetCell[c=2]{c}  v_3(t)\ge 3  &  &  E_9^d & \text{all} & 1 \\
& \SetCell[c=2,r=2]{c} v_3(t)=2   & & E_3^d &d\not\equiv 0\,(3) & 3/4 \\
&  & & E_9^d &d\equiv 0\,(3) & 1/4 \\
& \SetCell[c=2,r=2]{c} v_3(t)=1  &  & E_1^d& d\not\equiv 0\,(3) & 3/4\\ 
& & & E_3^d &d\equiv 0\,(3) &  1/4 \\
&  \SetCell[c=2]{c}  v_3(t)\leq 0 & &E_1^d & \text{all} & 1 \\
 \SetCell[r=2]{c}  L_3(25) &  \SetCell[c=2]{c}  v_5(t)\ge 1 &  &  E_{25}^d & \SetCell[r=2]{c}  \text{all} &\SetCell[r=2]{c}   1 \\
& \SetCell[c=2]{c}  v_5(t)\leq 0 &  &E_1^d & & \\
  \SetCell[r=2]{c}  L_4 &    \SetCell[c=2,r=2]{c} \text{genus $\geq 1$}   & & E_{3}^d & d\not\equiv 0\,(3) & 3/4 \\
& & &   E_{9}^d  &  d\equiv 0\,(3) & 1/4 \\
\pagebreak
\SetCell[r=8]{c} R_4(6) & 
 \SetCell[r=4]{c}  v_2(t)\geq 2 &  v_3(t)\geq 2 & E_6^d  & \text{all} &  1 \\
&  &  \SetCell[r=2]{c}v_3(t)=1 & E_2^d   &d\equiv 0\,(3)  & 1/4\\ 
&&   & E_6^d &d\not\equiv 0\,(3) &  3/4 \\
 &  &  v_3(t)\leq 0 & E_2^d  & \text{all} &  1 \\
 & \SetCell[r=4]{c}  v_2(t) \leq 1  &  v_3(t)\geq 2 & E_3^d & \text{all} &  1 \\
 &  &\SetCell[r=2]{c}  v_3(t) =1 & E_1^d   &d\equiv 0\,(3)   & 1/4\\ 
&&  & E_3^d &d\not\equiv 0\,(3) &  3/4 \\
& &  v_3(t)\leq 0 & E_1^d & \text{all} &  1 \\
\SetCell[r=12]{c} R_4(10)  & \SetCell[r=3]{c} v_2(t)>1 &  v_5(t)\neq 0 & \SetCell[r=2]{c} E_2^d & \SetCell[r=3]{c} \text{all} & \SetCell[r=3]{c} 1 \\
&  & \SetCell[r=1]{c}  \begin{array}{c}
 v_5(t)=0\\
 t\not\equiv 4\,(5)  \end{array}  & & & \\
 & & \SetCell[r=1]{c} 
   \begin{array}{c}
 v_5(t)=0\\
 t\equiv 4\,(5)  \end{array}  &
E_{10}^d 
&  & \\ 
& \SetCell[r=6]{c} v_2(t)=1 & \SetCell[r=2]{c}  v_5(t)\neq 0 & E_1^d     &d\equiv 0\,(2)   & 1/3\\ 
& &   &  E_2^d &d\not\equiv 0\,(2) &  2/3 \\
&  &\SetCell[r=2]{c}  \begin{array}{c}
 v_5(t)=0\\
 t\not\equiv 4\,(5)  \end{array}    & E_1^d     &d\equiv 0\,(2)   & 1/3\\ 
& &  &  E_2^d &d\not\equiv 0\,(2) &  2/3 \\
&  & \SetCell[r=2]{c} \SetCell[r=1]{c} 
   \begin{array}{c}
 v_5(t)=0\\
 t\equiv 4\,(5)  \end{array}  
& 
E_{5}^d    
      &d\equiv 0\,(2)  & 1/3\\ 
& & & E_{10}^d
&d\not\equiv 0\,(2) &  2/3 \\
 &  \SetCell[r=3]{c}  v_2(t)\le 0 & \SetCell[r=1]{c} 
   \begin{array}{c}
 v_5(t)=0\\
 t\equiv 4\,(5)  \end{array} 
  &
E_{5}^d  
& \SetCell[r=3]{c}    \text{all}  &  \SetCell[r=3]{c}  1 \\
& & v_5(t)\neq 0 & \SetCell[r=2]{c} E_1^d & & \\
  & & \begin{array}{c}
 v_5(t)=0\\
 t\not\equiv 4\,(5)  \end{array}  & & & \\
  \SetCell[r=2]{c}  R_4(14) &   \SetCell[c=2,r=2]{c}   \text{genus $\geq 1$}     & & E_{1}^d & d\not\equiv 0\,(7) & 7/8 \\
& & &  E_{7}^d  &  d\equiv 0\,(7) & 1/8 \\
  \SetCell[r=2]{c}  R_4(15) &   \SetCell[c=2,r=2]{c}   \text{genus $\geq 1$}     & & E_{1}^d & d\not\equiv 0\,(5) & 5/6 \\
& &  & E_{5}^d  &  d\equiv 0\,(5) & 1/6 \\
  \SetCell[r=2]{c}  R_4(21) &   \SetCell[c=2,r=2]{c}   \text{genus $\geq 1$}     & & E_{1}^d & d\not\equiv 0\,(3) & 3/4 \\
& &&  E_{3}^d  &  d\equiv 0\,(3) & 1/4 \\
\SetCell[r=4]{c} R_6  & \SetCell[r=2]{c}  v_2(t)>0 & v_3(t)\ne 0 &E_2^d &  \SetCell[r=4]{c}   \text{all}  & \SetCell[r=4]{c} 1\\
 &   & v_3(t)= 0  & E_{18}^d  & &  \\
&  \SetCell[r=2]{c} v_2(t)\le 0 & v_3(t)\ne 0  &E_1^d & &  \\
&  & v_3(t) = 0  & E_9^d  & &  \\
\pagebreak
\SetCell[r=9]{c} T_4 &\SetCell[c=2]{c}  v_2(t)\ge 6 & & E_4^d &  \text{all} & 1 \\
& \SetCell[c=2, r=2]{c} v_2(t)=5 & & E_4^d &d\equiv  0\,(2) & 1/3\\
& & & E_2^d & d\not\equiv  0\,(2)& 2/3\\
& \SetCell[c=2,r=2]{c} 
\begin{array}{c}
    v_2(t)=4  \\
    t/2^4\equiv 1\,(4)
\end{array} & &E_{12}^d & d\equiv  0\,(2) & 1/3\\
& & & E_{2}^d & d\not\equiv  0\,(2) & 2/3\\
&\SetCell[c=2]{c} \begin{array}{c}
    v_2(t)=4  \\
    t/2^4\equiv 3\,(4)
\end{array} & & E_{12}^d  &  \text{all}  & 1\\
& \SetCell[c=2,r=2]{c}
v_2(t)=3  & & E_{2}^d  & d\equiv  0\,(2) & 1/3\\
& & & E_{1}^d  & d\not\equiv  0\,(2) & 2/3\\
& \SetCell[c=2]{c} v_2(t)\le 2 & & E_{1}^d & \text{all}  & 1\\
\SetCell[r=4]{c}  T_6 & \SetCell[c=2]{c} v_2(t)\ge 3 & &E_{12}^d &\SetCell[r=4]{c}    \text{all} & \SetCell[r=4]{c}  1 \\
& \SetCell[c=2]{c}\begin{array}{c}
v_2(t)=2\\
t/2^2\equiv 3\,(4)
\end{array}
&& E_{8}^d  & &  \\
& \SetCell[c=2]{c}\begin{array}{c}
v_2(t)=2\\
t/2^2\equiv 1\,(4)
\end{array} &&
E_{22}^d  & &  \\
& \SetCell[c=2]{c} v_2(t)\le 1 && E_{1}^d  & &  \\
\SetCell[r=4]{c}  T_8 & \SetCell[c=2]{c}  v_2(t)\ge 2 & & E_{21}^d &  \SetCell[r=4]{c} \text{all} & \SetCell[r=4]{c} 1 \\
&  \SetCell[c=2]{c} \begin{array}{c}
v_2(t)=1\\
t/2\equiv 3\,(4)
\end{array}
& &
E_{81}^d  & &  \\
& \SetCell[c=2]{c}  \begin{array}{c}
v_2(t)=1\\
t/2\equiv 1\,(4)
\end{array} & &
E_{16}^d   & &  \\
&  \SetCell[c=2]{c} v_2(t)\le 0 & & E_{2}^d  & &  \\
 \SetCell[r=6]{c} S_8  & \SetCell[r=3]{c} v_3(t)>1 & v_2(t)\ne 0 & E_3^d &  \SetCell[r=6]{c}  \text{all}  & \SetCell[r=6]{c} 1 \\
& & 
\begin{array}{c}
v_2(t)=0\\[3pt]
t\equiv 3\,(4)
\end{array} 
& E_{12}^d & &  \\
& & \begin{array}{c}
v_2(t)=0\\[3pt]
t\equiv 1\,(4)
\end{array} 
 & E_{31}^d & &  \\
& \SetCell[r=3]{c}   v_3(t)\le 0 & v_2(t)\ne 0 & E_{1}^d & &  \\
& & 
\begin{array}{c}
v_2(t)=0\\[3pt]
t\equiv 3\,(4)
\end{array} 
& E_{4}^d  & &  \\
& & \begin{array}{c}
v_2(t)=0\\[3pt]
t\equiv 1\,(4)
\end{array} 
 & E_{21}^d   & &  \\

\end{longtblr}

\begin{proof}

For each possible type of $\mathbb{Q}$-isogeny graph, we follow the approach outlined in the examples in Section \ref{examples}. This allows us to construct, for each type, tables similar to those obtained there (Tables \ref{L39} and \ref{L211}), which contain the relevant information needed to establish Table \ref{table_theorem}. The corresponding data is summarized in Table \ref{proof}. For the complete details for each type, we refer to the supplementary material \cite{githubrepo}.

To complete the proof and determine the Faltings curve for each type, we follow a procedure similar to the one described in Section \ref{examples} (for the types $L_3(9)$ and $L_2(11)$). Let $\operatorname{Isog} = (E_{i_1}, \dots, E_{i_n})$, where $I=\{i_1,\dots,i_n\}$, $i_1=1$,  (see Table \ref{first} for the definition of $I$ for each type) denotes the $\mathbb{Q}$-isogeny class of elliptic curves associated with the graph $G$, and let $\text{type} = \operatorname{type}(G)$. Since the signatures are chosen so that the isogenies are normalized, there exist $v_i \in \mathbb{Q}$, $i\in I$, with $v_1 = 1$, such that, projectively, we have (see Table \ref{first}):  
\[
(\operatorname{vol}(E_i) \,:\, i\in I ) =(v_i\,:\, i\in I).
\]
Next, we obtain $[u(E)]$ and $[u(\Emind)]$ from Table \ref{proof} to determine which curve corresponds to the Faltings curve. If \( [u(E)] = (u_i\,:\, i\in I) \), then the volumes of the Néron lattices of \( (E_i\,:\, i\in I) \) satisfy  
\[
(\operatorname{vol}(\Emin_i)\,:\, i\in I) = (u_i^2 v_i \,:\, i\in I).
\]
Similarly, if $[u(\Emind)] = (\widetilde{u}_i\,:\, i\in I)$, then the volumes of the Néron lattices of $(\Emin_i^d\,:\, i\in I)$ satisfy  
\[
(\widetilde{u}_i^2 u_i^2 v_i \,:\, i\in I).
\]
Thus, the Faltings curve is the elliptic curve $E_k^d$ such that  
\[
\widetilde{u}_k^2 u_k^2 v_k = \max\{\widetilde{u}_i^2 u_i^2 v_i \mid  i\in I \}.
\]
\end{proof}

\begin{longtblr}
[caption = {Data for the proof},label=proof]
{cells = {mode=imath},hlines,vlines,measure=vbox,colspec  = ccccc,
hline{1,2,8,14,20,24,26,28,30,32,34,35,37,39,45,47,49,57,66,68,70,72,78,88,92,96,103 }={1.2pt,solid},
hline{50,52,53,58,62,63,73,75,76,97,100,101}={0.8pt,solid},
rowhead=1
}
 \text{type$(G)$} & t & [u(E)]  &  [u(\Emind)] &  d \\
 \SetCell[r=6]{c}  L_2(2) & \SetCell[r=1]{c} v_2(t)\ge 8 & \SetCell[r=1]{c} (1:2) & \SetCell[r=1]{c}(1:1) &  \text{all}\\
& \SetCell[r=1]{c} v_2(t)=7 & \SetCell[r=2]{c} (1:2) & \SetCell[r=1]{c} (1:1) & d\not\equiv 0\,(2)\\
& \SetCell[r=1]{c} \begin{array}{c}
 v_2(t)=6\\
v_2(t+64)\equiv 2,3(4)
 \end{array} & & \SetCell[r=1]{c} (2:1) & d\equiv 0\,(2) \\
&  \SetCell[r=1]{c} \begin{array}{c}
 v_2(t)=6\\
v_2(t+64)\equiv 0,1(4)
 \end{array}  & \SetCell[r=2]{c} (1:1) & \SetCell[r=1]{c} (1:1) & d\not\equiv 0\,(2)\\
& \SetCell[r=1]{c} v_2(t)=5 & & \SetCell[r=1]{c} (1:2) & d\equiv 0\,(2) \\
& \SetCell[r=1]{c}  v_2(t)\le 4 &\SetCell[r=1]{c} (1:1) & \SetCell[r=1]{c}(1:1) &  \text{all}\\
  \SetCell[r=6]{c}  L_2(3) &  v_3(t)\ge 5 & \SetCell[r=1]{c} (1:3) & \SetCell[r=1]{c}(1:1) &  \text{all}\\
&  v_3(t)=4 & \SetCell[r=2]{c} (1:3) & \SetCell[r=1]{c} (1:1) & d\not\equiv 0\,(3) \\
&  \begin{array}{c}
 v_3(t)=3\\
 v_3(t+27)\equiv 3,4,5\,(6)
 \end{array} & & \SetCell[r=1]{c} (3:1) & d\equiv 0\,(3) \\
 & \begin{array}{c}
 v_3(t)=3\\
 v_3(t+27)\equiv 0,1,2\,(6)
 \end{array}  & \SetCell[r=2]{c} (1:1) & \SetCell[r=1]{c} (1:1) & d\not\equiv 0\,(3)\\
 & v_3(t)=2 & & \SetCell[r=1]{c} (1:3) & d\equiv 0\,(3) \\
 & v_3(t)\le 1 &\SetCell[r=1]{c} (1:1) & \SetCell[r=1]{c}(1:1) & \text{all}\\
 \SetCell[r=6]{c}  L_2(5) &   \SetCell[r=1]{c} v_5(t)\ge 3 & \SetCell[r=1]{c} (1:5) & \SetCell[r=1]{c} (1:1) & \text{all} \\
& \SetCell[r=2]{c}  v_5(t)=2 & \SetCell[r=2]{c} (1:5) & \SetCell[r=1]{c} (1:1) & d\not\equiv 0\,(5)\\
& & & \SetCell[r=1]{c} (5:1) & d\equiv 0\,(5)\\
 & \SetCell[r=2]{c}  v_5(t)=1 & \SetCell[r=2]{c} (1:1) & \SetCell[r=1]{c} (1:1) & d\not\equiv 0\,(5) \\
& & & \SetCell[r=1]{c} (1:5) & d\equiv 0\,(5) \\
& \SetCell[r=1]{c} v_5(t)\le 0 & \SetCell[r=1]{c} (1:1) & \SetCell[r=1]{c}  (1:1) & \text{all} \\
 \pagebreak
 \SetCell[r=4]{c}  L_2(7) &  \SetCell[r=1]{c} v_7(t)\ge 2 & \SetCell[r=1]{c} (1:7) & \SetCell[r=1]{c} (1:1) & \text{all}  \\
& \SetCell[r=2]{c}  v_7(t)=1 & \SetCell[r=2]{c} (1:1) & \SetCell[r=1]{c} (1:1) & d\not\equiv 0\,(7)\\
& & & \SetCell[r=1]{c} (1:7) & d\equiv 0\,(7) \\
& \SetCell[r=1]{c} v_7(t)\le 0 & \SetCell[r=1]{c} (1:1) & \SetCell[r=1]{c}  (1:1) & \text{all}\\
\SetCell[r=2]{c}  L_2(11) & \SetCell[r=2]{c}   \text{genus $\geq 1$} & \SetCell[r=2]{c}  (1:1)  &  (1:1) & d\not\equiv 0\,(11)\\
& &   &  (1:11) & d\equiv 0\,(11) \\
\SetCell[r=2]{c}  L_2(13) &  v_{13}(t)>0 & (1:13) &(1:1) & \SetCell[r=2]{c} \text{all} \\
& v_{13}(t)\le 0 & (1:1) &  (1:1) &  \\
\SetCell[r=2]{c}  L_2(17) & \SetCell[r=2]{c}    \text{genus $\geq 1$}&  \SetCell[r=2]{c} (1:1)  &  (1:1) & d\not\equiv 0\,(17)\\
& &     &  (1:17) & d\equiv 0\,(17) \\
\SetCell[r=2]{c}  L_2(19) &  \SetCell[r=2]{c}    \text{genus $\geq 1$}& \SetCell[r=2]{c}  (1:1)  &  (1:1) & d\not\equiv 0\,(19)\\
& &    &  (1:19) & d\equiv 0\,(19) \\
\SetCell[r=1]{c}  L_2(37) &  \SetCell[r=1]{c}    \text{genus $\geq 1$} & (1:1)  &  (1:1) & \text{all}\\
\SetCell[r=2]{c}  L_2(43) & \SetCell[r=2]{c}    \text{genus $\geq 1$}& \SetCell[r=2]{c}  (1:1)  &  (1:1) & d\not\equiv 0\,(43)\\
& &     &  (1:43) & d\equiv 0\,(43) \\
\SetCell[r=2]{c}  L_2(67) & \SetCell[r=2]{c}    \text{genus $\geq 1$}& \SetCell[r=2]{c}  (1:1)  &  (1:1) & d\not\equiv 0\,(67)\\
& &    &  (1:67) & d\equiv 0\,(67) \\
\SetCell[r=2]{c}  L_2(163) & \SetCell[r=2]{c}    \text{genus $\geq 1$}&  \SetCell[r=2]{c} (1:1)  &  (1:1) & d\not\equiv 0\,(163)\\
& &   &  (1:163) & d\equiv 0\,(163) \\
\SetCell[r=6]{c}  L_3(9) & v_3(t)\ge 3 & (1:3:3^{2}) & (1:1:1) & \text{all} \\
& \SetCell[r=2]{c}  v_3(t)=2 & \SetCell[r=2]{c} (1:3:3) & (1:1:1) & d\not\equiv 0\,(3) \\
& & & \SetCell[r=1]{c} (1:1:3) & d\equiv 0\,(3) \\
& \SetCell[r=2]{c} v_3(t)=1 & \SetCell[r=2]{c} (1:1:1) & (1:1:1) & d\not\equiv 0\,(3) \\
& & & \SetCell[r=1]{c} (1:3:3) & d\equiv 0\,(3) \\
 & \SetCell[r=1]{c} v_3(t)\le 0 & \SetCell[r=1]{c} (1:1:1) & (1:1:1) & \text{all} \\
\SetCell[r=2]{c}  L_3(25) & v_5(t)\ge 1 & (1:5:5^{2}) & (1:1:1) &  \SetCell[r=2]{c} \text{all} \\
&  v_5(t)\le 0 & (1:1:1) & (1:1:1) &  \\
\SetCell[r=2]{c}  L_4 & \SetCell[r=2]{c}    \text{genus $\geq 1$}& \SetCell[r=2]{c}  (1:1:1:1)  &  \left(1:1:1:1\right)  & d\not\equiv 0\,(3)\\
& &   & \left(1:1:3:3\right) & d\equiv 0\,(3) \\
\pagebreak
\SetCell[r=8]{c} R_4(6)&  &[u_2(E)] & [u_2(\Emind)]  &\\
& \SetCell[r=1]{c} v_2(t)\geq 2 &  (1:2:1:2) & (1:1:1:1) &   \SetCell[r=2]{c} \text{all} \\
& \SetCell[r=1]{c} v_2(t)\leq 1 & (1:1:1:1) & (1:1:1:1) &   \\
& & [u_3(E)] & [u_3(\Emind)] &\\
& \SetCell[r=1]{c} v_3(t)\ge 2 & (1:1:3:3) & (1:1:1:1) &   \text{all} \\
& \SetCell[r=2]{c} v_3(t)=1 & \SetCell[r=2]{c} (1:1:1:1) & (1:1:1:1) & d\not\equiv 0\,(3)   \\
& &  & (1:1:3:3) & d\equiv 0\,(3)  \\
& \SetCell[r=1]{c} v_3(t)\leq 0 & (1:1:1:1) & (1:1:1:1) & \text{all}  \\
 \SetCell[r=9]{c} R_4(10) &  &[u_2(E)] & [u_2(\Emind)] &\\
 &   v_2(t)>1 &  (1:2:1:2) & (1:1:1:1) &   \text{all} \\
& \SetCell[r=2]{c} v_2(t)=1 & \SetCell[r=2]{c} (1:2:1:2) & (1:1:1:1) & d\not\equiv 0\,(2)   \\
& &  &  (2:1:2:1) & d\equiv 0\,(2)  \\
 & v_2(t)\leq 0 &  (1:1:1:1) & (1:1:1:1) &   \text{all}\\
 &  & [u_5(E)] & [u_5(\Emind)]  &\\
& v_5(t)\ne 0 & \SetCell[r=2]{c} (1:1:1:1) & \SetCell[r=2]{c} (1:1:1:1) &   \SetCell[r=2]{c} \text{all} \\
  &   \begin{array}{c}
 v_5(t)=0\\
 t\not \equiv 4\,(5)  \end{array} & & & \\
 &  \begin{array}{c}
 v_5(t)=0\\
 t\equiv 4\,(5)   \end{array} &  (1:1:5:5) &  (1:1:1:1) &  \text{all}\\
\SetCell[r=2]{c}  R_4(14) & \SetCell[r=2]{c}    \text{genus $\geq 1$}& \SetCell[r=2]{c}  (1:1:1:1)  &  (1:1:1:1)  & d\not\equiv 0\,(7)\\
& &   & (1:1:7:7) & d\equiv 0\,(7) \\
\SetCell[r=2]{c}  R_4(15) & \SetCell[r=2]{c}   \text{genus $\geq 1$} & \SetCell[r=2]{c}  (1:1:1:1)  & (1:1:1:1)   & d\not\equiv 0\,(5)\\
& &   & (1:1:5:5)  & d\equiv 0\,(5) \\
\SetCell[r=2]{c}  R_4(21) & \SetCell[r=2]{c}   \text{genus $\geq 1$} & \SetCell[r=2]{c}  (1:1:1:1)  & (1:1:1:1)  & d\not\equiv 0\,(3)\\
& &   &(1:3:1:3) & d\equiv 0\,(3) \\
\SetCell[r=6]{c} R_6 & & [u_2(E)] & [u_2(\Emind)] & \\
& \SetCell[c=1]{c} v_2(t)>0 & \SetCell[r=1]{c} (1:2:1:2:1:2) & (1:\cdots :1) & \SetCell[r=2]{c}\text{all} \\
& \SetCell[r=1]{c} v_2(t)\leq 0 & \SetCell[r=1]{c} (1:1:1:1:1:1) & (1:\cdots :1)&  \\
&  & [u_3(E)] & [u_3(\Emind)] & \\
& \SetCell[c=1]{c} v_3(t)=0 & \SetCell[r=1]{c} (1:1:3:3:3^2:3^2) &(1:\cdots :1)&    \SetCell[r=2]{c}\text{all} \\
& \SetCell[r=1]{c} v_3(t)\ne 0 & \SetCell[r=1]{c} (1:1:1:1:1:1) & (1:\cdots :1)&  \\
 \pagebreak
\SetCell[r=10]{c} T_4 & v_2(t)\ge 6  & \SetCell[r=1]{c} (1:2:2^2:1) & (1:1:1:1) &  \text{all}\\
& \SetCell[r=2]{c} v_2(t)=5 &  \SetCell[r=2]{c} (1:2:2:2) & (1:1:1:1) & d\not\equiv 0\,(2) \\
& & & \SetCell[r=1]{c} (1:1:2:1) & d\equiv  0\,(2)\\
 & \SetCell[r=2]{c}
\begin{array}{c}
v_2(t)=4  \\
t/2^4\equiv 1\,(4)
\end{array}
& \SetCell[r=2]{c} (1:2:2:2) & (1:1:1:1) & d\not\equiv  0\,(2)\\
& &  & \SetCell[r=1]{c} (1:1:1:2) & d\equiv 0\,(2)\\
& \SetCell[r=2]{c}
\begin{array}{c}
v_2(t)=4  \\
t/2^4\equiv 3\,(4)
\end{array}
 & \SetCell[r=2]{c} (1:2:2:2^2) & \SetCell[r=2]{c} (1:1:1:1) & \SetCell[r=2]{c}\text{all} \\
 & & & \\
& \SetCell[r=2]{c} v_2(t)=3  & \SetCell[r=2]{c} (1:1:1:1) & (1:2:2:2) & d\not\equiv  0\,(2)\\
& &  & \SetCell[r=1]{c} (1:1:1:1) & d\equiv 0\,(2) \\
& v_2(t)\le 2  & \SetCell[r=1]{c} (1:1:1:1) & (1:1:1:1) & \text{all}\\
\SetCell[r=4]{c} T_6 &  v_2(t)\ge 3 & \SetCell[r=1]{c} (1:2:2^2:2:2:2) & (1:\cdots :1)&  \SetCell[r=4]{c}\text{all}\\
& \SetCell[r=1]{c} \begin{array}{c}
v_2(t)=2\\
t/2^2\equiv 3\,(4)
\end{array}  & \SetCell[r=1]{c} (1:2:2:2^2:2^2:2^3) &(1:\cdots :1) &  \\
  &\SetCell[r=1]{c} \begin{array}{c}
v_2(t)=2\\
t/2^2\equiv 1\,(4)
\end{array}  & \SetCell[r=1]{c} (1:2:2:2^2:2^3:2^2) & (1:\cdots :1) &   \\
 & \SetCell[r=1]{c} v_2(t)\le 1 & \SetCell[r=1]{c} (1:1:1:1:1:1) & (1:\cdots :1) & \\
\SetCell[r=4]{c} T_8 &  v_2(t)\ge 2 & \SetCell[r=1]{c} (1:2:2^2:2:2:2:2:2) & (1:\cdots :1)&  \SetCell[r=4]{c}\text{all}\\
& \SetCell[r=1]{c} \begin{array}{c}
v_2(t)=1\\
t/2\equiv 3\,(4)
\end{array}  & \SetCell[r=1]{c} (1:2:2:2^2:2^2:2^3:2^4:2^3) & (1:\cdots :1)&   \\
& \SetCell[r=1]{c} \begin{array}{c}
v_2(t)=1\\
t/2\equiv 1\,(4)
\end{array}  & \SetCell[r=1]{c} (1:2:2:2^2:2^2:2^3:2^3:2^4) & (1:\cdots :1) &   \\
& \SetCell[r=1]{c} v_2(t)\le 0 & \SetCell[r=1]{c} (1:1:1:1:1:1:1:1) & (1:\cdots :1)&   \\
\SetCell[r=7]{c} S_8 & &   [u_2(E)] & [u_2(\Emind)] &  \\
& v_2(t)\ne 0& \SetCell[r=1]{c} (1:1:1:1:1:1:1:1) & (1:\cdots :1) & \SetCell[r=3]{c}\text{all}\\
& \SetCell[r=1]{c} \begin{array}{c}
v_2(t)=0\\
t/2\equiv 3\,(4)
\end{array}  & \SetCell[r=1]{c} (1:1:2:2:2:2^2:2^2:2) &(1:\cdots :1) & \\
& \SetCell[r=1]{c} \begin{array}{c}
v_2(t)=0\\
t/2\equiv 1\,(4)
\end{array}  & \SetCell[r=1]{c} (1:1:2:2:2^2:2:2:2^2)  & (1:\cdots :1)&  \\
&  & [u_3(E)] & [u_3(\Emind)] & \\
&\SetCell[r=1]{c} v_3(t)>0& \SetCell[r=1]{c} (1:3:1:3:1:3:1:3) & (1:\cdots :1)& \SetCell[r=2]{c}\text{all}\\
& \SetCell[r=1]{c} v_3(t)\le 0& \SetCell[r=1]{c} (1:1:1:1:1:1:1:1) & (1:\cdots :1) &\\
\end{longtblr}

\newpage

\begin{longtblr}
[caption = {Data for the proof (first volumes)},label=first]
{cells = {mode=dmath},hlines,vlines,measure=vbox,colspec  = cll,rowhead=1
}
 \text{type$(G)$} &  \operatorname{Isog} & (\operatorname{vol}(E)\,:\,E\in\operatorname{Isog}) \\
  \SetCell[r=2]{c} L_2(p) & (E_1,E_p) & \left(1:\frac{1}{p}\right)\\ 
  &   \SetCell[c=2]{l} p=2,3,5,7,11,13,17,19,37,43,67,163\\
  \SetCell[r=2]{c}  L_3(p^2)& (E_1,E_p,E_{p^2}) & \left(1:\frac{1}{p}:\frac{1}{p^2}\right)\\ 
  &  \SetCell[c=2]{l} p=3,5  \\
  L_4 & (E_1,E_3,E_{9},E_{27}) & \left(1:\frac{1}{3}:\frac{1}{9}:\frac{1}{27}\right)  \\
   \SetCell[r=2]{c}  R_4(pq) &  (E_1,E_p,E_q,E_{pq}) & \left(1:\frac{1}{p}:\frac{1}{q}:\frac{1}{pq}\right)\\
   &  \SetCell[c=2]{l} (p,q)=(2,3),(2,5),(2,7),(3,5),(3,7)  \\  
  R_6 & (E_1,E_2,E_3,E_{6},E_9,E_{18}) &  \left(1:\frac{1}{2}:\frac{1}{3}:\frac{1}{6}:\frac{1}{9}:\frac{1}{18}\right)  \\
  T_4 &  (E_1,E_2,E_4,E_{12}) & \left(1:\frac{1}{2}:\frac{1}{4}:\frac{1}{4}\right)  \\
  T_6 &  (E_1,E_2,E_{12},E_4,E_8,E_{12}) & \left(1:\frac{1}{2}:\frac{1}{4}:\frac{1}{4}:\frac{1}{8}:\frac{1}{8}\right)  \\
  T_8 &  (E_1,E_2,E_{21},E_4,E_{41},E_8,E_{81},E_{16}) & \left(1:\frac{1}{2}:\frac{1}{4}:\frac{1}{4}:\frac{1}{8}:\frac{1}{8}:\frac{1}{16}:\frac{1}{16}\right)  \\
  S_8 &  (E_1,E_3,E_2,E_6,E_{21},E_{12},E_{4},E_{31})  & \left(1:\frac{1}{3}:\frac{1}{2}:\frac{1}{6}:\frac{1}{4}:\frac{1}{12}:\frac{1}{4}:\frac{1}{12}\right)  \\
\end{longtblr}


\end{document}